\documentclass[11pt]{article}

\date{}
\title{Discrete dynamical systems in group theory}
\author{Dikran Dikranjan \and Anna Giordano Bruno}

\usepackage{vmargin}
\setmarginsrb{28mm}{20mm}{28mm}{25mm}%
          {10mm}{7mm}{10mm}{10mm}

\usepackage[english]{babel}
\usepackage{amsmath,amsfonts,amssymb,amsthm}
\usepackage{latexsym}
\usepackage{hyperref}

\numberwithin{equation}{section}

\newtheorem{Definition}{Definition}[section]
\newtheorem{Example}[Definition]{Example}
\newtheorem{Lemma}[Definition]{Lemma}
\newtheorem{Theorem}[Definition]{Theorem}
\newtheorem{Remark}[Definition]{Remark}
\newtheorem{Proposition}[Definition]{Proposition}
\newtheorem{Corollary}[Definition]{Corollary}
\newtheorem{fact}[Definition]{Fact}
\newtheorem{question}[Definition]{Question}
\newtheorem{problem}[Definition]{Problem}

\input xy
\xyoption{all}

\def\mod{\mathbf{Mod}}
\def\id{\mathrm{id}}

\def\V{{\cal V}}
\def\W{{\cal W}}
\def\sF{{\cal F}}

\def\sH{{\cal H}}
\def\La{{\cal L}}
\def\sC{{\cal C}}
\def\AG{{\mathbf{AbGrp}}}
\def\CT{{\mathbf{CTop}}}

\def\ent{{\mathrm{ent}}}
\def\CAG{{\mathbf{CAbGrp}}}
\def\MS{{\mathbf{MesSp}}}

\def\cov{\mathfrak{cov}}

\newcommand{\N}{\mathbb N}
\newcommand{\Z}{\mathbb Z}
\newcommand{\Q}{\mathbb Q}
\newcommand{\R}{\mathbb R}

\def\SL{{\mathfrak{L}}}
\def\Se{{\mathfrak{S}}}

\def\ent{\mathrm{ent}}

\def\sM{\mathcal{H}}

\def\U{{\cal U}}
\def\f{\phi}

\def\V{{\cal V}}
\def\W{{\cal W}}
\def\End{\mathrm{End}}

\begin{document}

\maketitle

\abstract{
In this expository paper we describe the unifying approach for many known entropies in Mathematics developed in \cite{DGV1}.

First we give the notion of semigroup entropy $h_\Se:\Se\to\R_+$ in the category $\Se$ of normed semigroups and contractive homomorphisms, recalling also its properties from \cite{DGV1}. 
For a specific category $\mathfrak X$ and a functor $F:\mathfrak X\to \Se$ we have the entropy $h_F$, defined by the composition $h_F=h_\Se\circ F$, which automatically satisfies the same properties proved for $h_\Se$.
This general scheme permits to obtain many of the known entropies as $h_F$, for appropriately chosen categories $\mathfrak X$ and functors $F:\mathfrak X\to \Se$.

In the last part we recall the definition and the fundamental properties of the algebraic entropy for group endomorphisms, noting how its deeper properties depend on the specific setting. Finally we discuss the notion of growth for flows of groups, comparing it with the classical notion of growth for finitely generated groups.}




\section{Introduction} 

This paper covers the series of three talks given by the first named author at the conference ``Advances in Group Theory and Applications 2011'' held in June, 2011 in Porto Cesareo. It is a survey about entropy in Mathematics, the approach is the categorical one adopted in \cite{DGV1} (and announced in \cite{D}, see also \cite{LoBu}).

\medskip
We start Section 4 recalling that a \emph{flow} in a category $\mathfrak X$ is a pair $(X,\phi)$, where $X$ is an object of $\mathfrak X$ and $\phi: X\to X$ is a morphism in $\mathfrak X$.
A morphism between two flows $\phi: X\to X$ and $\psi: Y\to Y$ is a morphism $\alpha: X \to Y$ in $\mathfrak X$ such that  the diagram
$$\xymatrix{X\ar[r]^{\alpha} \ar[d]_{\phi}&Y\ar[d]^{\psi}\\X\ar[r]^{\alpha}& Y.}$$
commutes. This defines the category $\mathbf{Flow}_{\mathfrak X}$ of flows in $\mathfrak X$. 

To classify flows in $\mathfrak X$ up to isomorphisms one uses invariants, and entropy is roughly a numerical invariant associated to flows. 
Indeed, letting $\R_{\geq 0} = \{r\in \R: r\geq 0\}$  and $\R_+= \R_{\geq 0}\cup \{\infty\}$, by the term \emph{entropy} we intend a function 
\begin{equation}\label{dag}
h: \mathbf{Flow}_{\mathfrak X}\to \R_+,
\end{equation}
obeying the invariance law $h(\phi) = h(\psi)$ whenever $(X,\phi)$ and $(Y,\psi)$ are isomorphic flows.
The value $h(\phi)$ is supposed to measure the degree to which $X$ is ``scrambled" by $\phi$,  so for example an entropy should assign $0$ to all identity maps. 
For simplicity and with some abuse of notations, we adopt the following

\medskip
\noindent\textbf{Convention.}
If $\mathfrak X$ is a category and $h$ an entropy of $\mathfrak X$, writing $h: {\mathfrak X}\to \R_+$ we always mean $h: \mathbf{Flow}_{\mathfrak X}\to \R_+$ as in \eqref{dag}.
\medskip

The first notion of entropy in Mathematics was the measure entropy $h_{mes}$ introduced by Kolmogorov \cite{K} and Sinai \cite{Sinai} in 1958 in Ergodic Theory. The topological entropy $h_{top}$ for continuous self-maps of compact spaces was defined by Adler, Konheim and McAndrew \cite{AKM} in 1965. Another notion of topological entropy $h_B$ for uniformly continuous self-maps of metric spaces
was given later by Bowen \cite{B} (it coincides with $h_{top}$ on compact metric spaces). Finally, entropy was taken also in Algebraic Dynamics by Adler, Konheim and McAndrew \cite{AKM} in 1965 and Weiss \cite{W} in 1974; they defined an entropy $\ent$ for endomorphisms of torsion abelian groups. Then Peters \cite{P} in 1979 introduced its extension $h_{alg}$ to automorphisms of abelian groups; finally $h_{alg}$ was defined in \cite{DG} and \cite{DG-islam} for any group endomorphism. Recently also a notion of algebraic entropy for module endomorphisms was introduced in \cite{SZ}, namely the algebraic $i$-entropy $\ent_i$, where $i$ is an invariant of a module category. Moreover, the adjoint algebraic entropy $\ent^\star$ for group endomorphisms was investigated in \cite{DGV} (and its topological extension in \cite{G}). Finally, one can find in \cite{AZD} and \cite{DG-islam} two ``mutually dual'' notions of entropy for self-maps of sets, namely the covariant set-theoretic entropy $\mathfrak h$ and the contravariant set-theoretic entropy $\mathfrak h^*$.

The above mentioned  specific entropies determined the choice of the main cases considered  in this paper. Namely, $\mathfrak X$ will be one of the following categories (other examples can be found in \S\S \ref{NewSec1} and \ref{NewSec2}): 
\begin{itemize}
\item[(a)] $\mathbf{Set}$ of sets and maps and its non-full subategory $\mathbf{Set}_{\mathrm{fin}}$ of sets and finite-to-one maps (set-theoretic entropies $\mathfrak h$ and $\mathfrak h^*$ respectively);
\item[(b)] $\mathbf{CTop}$ of compact topological spaces and continuous maps (topological entropy $h_{top}$);
\item[(c)] $\mathbf{Mes}$ of probability measure spaces and measure preserving maps (measure entropy $h_{mes}$);
\item[(d)] $\mathbf{Grp}$ of groups and group homomorphisms and its subcategory $\mathbf{AbGrp}$ of abelian groups (algebraic entropy $\ent$, algebraic entropy $h_{alg}$ and adjoint algebraic entropy $\ent^\star$);
\item[(e)] $\mathbf{Mod}_R$ of right modules over a ring $R$ and $R$-module homomorphisms (algebraic $i$-entropy $\ent_i$). 
\end{itemize}
Each of these entropies has its specific definition, usually given by limits computed on some ``trajectories'' and by taking the supremum of these quantities (we will see some of them explicitly). 
The proofs of the basic properties take into account the particular features of the specific categories in each case too. It appears that all these definitions and basic properties share a lot of common features. The aim of our approach is to unify them in some way, starting from a general notion of entropy of an appropriate category. This will be the semigroup entropy $h_\Se$ defined on the category $\Se$ of normed semigroups. 

\medskip
In Section \ref{sem-sec} we first introduce the category $\Se$ of normed semigroups and related basic notions and examples mostly coming from \cite{DGV1}. Moreover, in \S \ref{preorder-sec} (which can be avoided at a first reading) we add a preorder to the semigroup and discuss the possible behavior of a semigroup norm with respect to this preorder. Here we include also the subcategory $\mathfrak L$ of $\Se$ of normed semilattices, as 
the functors given in Section \ref{known-sec} often have as a target actually a normed semilattice. 
 
In \S \ref{hs-sec} we define explicitly the semigroup entropy $h_\Se: \Se \to \R_+$ on the category $\Se$ of normed semigroups. Moreover we list all its basic properties, clearly inspired by those of the known entropies, such as Monotonicity for factors, Invariance under conjugation, Invariance under inversion, Logarithmic Law, Monotonicity for subsemigroups, Continuity for direct limits, Weak Addition Theorem and Bernoulli normalization.

\medskip
Once defined the semigroup entropy $h_\Se:\Se\to \R_+$, our aim is to obtain all known entropies $h:{\mathfrak X} \to \R_+$ as a composition $h_F:=h_\Se \circ F$ of a functor $F: \mathfrak X\to \Se$ and $h_\Se$:
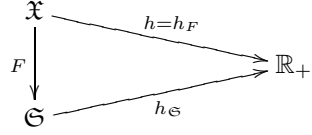
\begin{equation*}
\xymatrix@R=6pt@C=37pt
{\mathfrak{X}\ar[dd]_{F}\ar[rrd]^{h=h_F}	&	&	\\
& &	\R_+		\\
\Se\ar[rru]_{h_\Se}						&	&
}
\end{equation*}
This is done explicitly in Section \ref{known-sec}, where all specific entropies listed above are obtained in this scheme. We dedicate to each of them a subsection, each time giving explicitly the functor from the considered category to the category of normed semigroups. More details and complete proofs can be found in \cite{DGV1}.
These functors and the entropies are summarized by the following diagram:
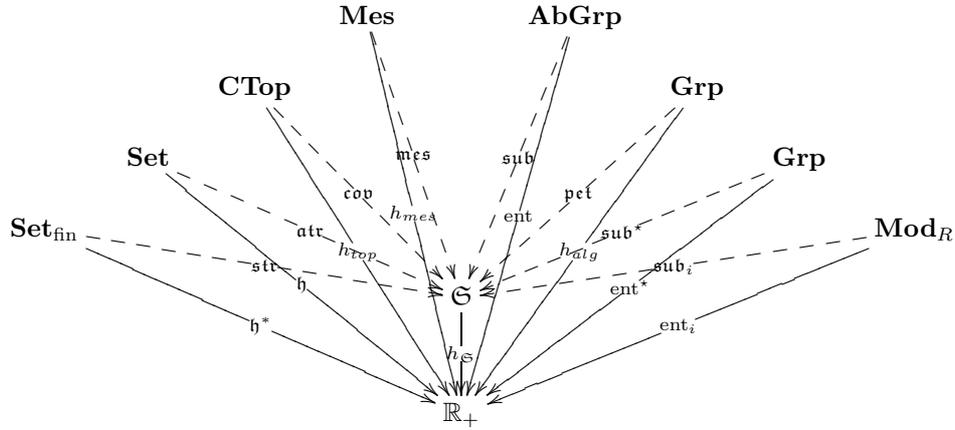
\begin{equation*}
\xymatrix@-1pc{
&&&\mathbf{Mes}\ar@{-->}[ddddr]|-{\mathfrak{mes}}\ar[ddddddr]|-{h_{mes}}& &\mathbf{AbGrp}\ar[ddddddl]|-{\mathrm{ent}}\ar@{-->}[ddddl]|-{\mathfrak{sub}}&&\\
& &\mathbf{CTop}\ar@{-->}[dddrr]|-{\mathfrak{cov}}\ar[dddddrr]|-{h_{top}}&  & & & \mathbf{Grp}\ar@{-->}[dddll]|-{\mathfrak{pet}}\ar[dddddll]|-{h_{alg}} & &\\
& \mathbf{Set}\ar@{-->}[ddrrr]|-{\mathfrak{atr}}\ar[ddddrrr]|-{\mathfrak h} && & &	&  & \mathbf{Grp}\ar@{-->}[ddlll]|-{\mathfrak{sub}^\star}\ar[ddddlll]|-{\ent^\star} \\
\mathbf{Set}_\mathrm{fin}\ar@{-->}[drrrr]|-{\mathfrak{str}}\ar[dddrrrr]|-{\mathfrak h^*} && && & &	&  &\mathbf{Mod}_R\ar@{-->}[dllll]|-{\mathfrak{sub}_i}\ar[dddllll]|-{\ent_i} \\
& && & \mathfrak S \ar[dd]|-{h_\mathfrak S} && & \\
 \\
& && &{\R_+}	 &	& &
}		
\end{equation*}
In this way we obtain a simultaneous and uniform definition of all entropies and uniform proofs (as well as a better understanding) of their general properties, namely the basic properties of the specific entropies can be derived directly from those proved for the semigroup entropy.

\medskip
The last part of Section \ref{known-sec} is dedicated to what we call Bridge Theorem (a term coined by L. Salce), that is roughly speaking a connection between entropies  $h_1:\mathfrak X_1 \to \R_+$ and $h_2:\mathfrak X_2 \to \R_+$ via functors $\varepsilon: \mathfrak X_1 \to \mathfrak X_2$. Here is a formal definition of this concept:

\begin{Definition}\label{BTdef}
Let $\varepsilon: \mathfrak X_1 \to \mathfrak X_2$ be a functor and let $h_1:\mathfrak X_1 \to \R_+$ and $h_2:\mathfrak X_2 \to \R_+$ be entropies of the categories $\mathfrak X_1 $ and $ \mathfrak X_2$, respectively (as in the diagram below).
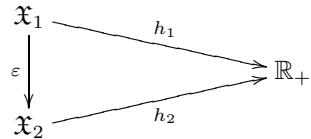
\begin{equation*}\label{Buz}
\xymatrix@R=6pt@C=37pt
{\mathfrak{X}_1\ar[dd]_{\varepsilon}\ar[rrd]^{h_{1}}	&	&	\\
& &	\R_+		\\
\mathfrak{X}_2\ar[rru]_{h_{2}}						&	&
}
\end{equation*}We say that the pair $(h_1, h_2)$ satisfies the \emph{weak Bridge Theorem} with respect to the functor $\varepsilon$ 
if there exists a positive constant $C_\varepsilon$, such that for every endomorphism $\f$ in $\mathfrak X_1$
\begin{equation}\label{sBT}
h_2(\varepsilon(\f)) \leq C_\varepsilon h_1(\f).
\end{equation}
If equality holds in \eqref{sBT} we say that $(h_1,h_2)$ satisfies the \emph{Bridge Theorem} with respect to $\varepsilon$, and we shortly denote this by $(BT_\varepsilon)$.
\end{Definition}

In \S \ref{BTsec} we discuss the Bridge Theorem passing through the category $\Se$ of normed semigroups and so using the new semigroup entropy. This approach permits for example to find a new and transparent proof of Weiss Bridge Theorem (see Theorem \ref{WBT}) as well as for other Bridge Theorems.

\medskip
A first limit of this very general setting is the loss of some of the deeper properties that a specific entropy may have.
So in the last Section \ref{alg-sec} for the algebraic entropy we recall the definition and the fundamental properties, which cannot be deduced from the general scheme.

We start Section 4 recalling the Algebraic Yuzvinski Formula (see Theorem \ref{AYF}) recently proved in \cite{GV}, giving the values of the algebraic entropy of linear transformations of finite-dimensional rational vector spaces in terms of the Mahler measure. In particular, this theorem provides a connection of the algebraic entropy with the famous Lehmer Problem. Two important applications of the Algebraic Yuzvinski Formula are the Addition Theorem and the Uniqueness Theorem for the algebraic entropy in the context of abelian groups.

\medskip
In \S \ref{Growth-sec} we describe the connection of the algebraic entropy with the classical topic of growth of finitely generated groups  in Geometric Group Theory. Its definition was given  independently by Schwarzc \cite{Sch} and Milnor \cite{M1}, and after the publication of \cite{M1} it was intensively investigated; several fundamental results were obtained by Wolf \cite{Wolf}, Milnor \cite{M2}, Bass \cite{Bass}, Tits \cite{Tits} and Adyan \cite{Ad}. In \cite{M3} Milnor proposed his famous problem (see Problem \ref{Milnor-pb} below); the question about the existence of finitely generated groups with intermediate growth was answered positively by Grigorchuk in \cite{Gri1,Gri2,Gri3,Gri4}, while the characterization of finitely generated groups with polynomial growth was given by Gromov in \cite{Gro} (see Theorem \ref{GT}).

Here we introduce the notion of finitely generated flows $(G,\phi)$ in the category of groups and define the growth of $(G,\phi)$. When $\phi=\id_G$ is the identical endomorphism, then $G$ is a finitely generated group and we find exactly the classical notion of growth. In particular we recall a recent significant result from \cite{DG0} extending Milnor's dichotomy (between polynomial and exponential growth) to finitely generated flows in the abelian case (see Theorem \ref{DT}). We leave also several open problems and questions about the growth of finitely generated flows of groups.

The last part of the section, namely \S \ref{aent-sec}, is dedicated to the adjoint algebraic entropy. As for the algebraic entropy, we recall its original definition and its main properties, which cannot be derived from the general scheme. In particular, the adjoint algebraic entropy can take only the values $0$ and $\infty$ (no finite positive value is attained) and we see that the Addition Theorem holds only restricting to bounded abelian groups.

\medskip
A natural side-effect of the wealth of nice properties of the entropy $h_F=h_\Se\circ F$, obtained from the semigroup entropy $h_\Se$ through functors $F:\mathfrak X\to \Se$, is the loss of some entropies that do not have all these properties.
For example Bowen's entropy $h_B$ cannot be obtained as $h_F$ since $h_B(\phi^{-1})= h_B(\phi)$ fails even for the automorphism $\phi: \R \to \R$ defined by $\phi(x)= 2x$, see \S \ref{NewSec2} for an extended comment on this issue; there we also discuss the possibility to obtain Bowen's topological entropy of measure preserving topological automorphisms of locally compact groups in the framework of our approach. For the same reason other entropies that cannot be covered by this approach are the intrinsic entropy for endomorphisms of abelian groups \cite{DGSV} and the topological entropy for automorphisms of locally compact totally disconnected groups \cite{DG-tdlc}. This occurs also for the function $\phi \mapsto \log s(\phi)$, where $s(\phi)$ is the scale function defined by Willis \cite{Willis,Willis2}.
The question about the relation of the scale function to the algebraic or topological entropy was posed by T. Weigel at the conference; these non-trivial relations are discussed for the topological entropy in \cite{BDG}.

\section{The semigroup entropy}\label{sem-sec}

\subsection{The category $\Se$ of normed semigroups}

We start this section introducing the category $\Se$ of normed semigroups, and other notions that are fundamental in this paper.

\begin{Definition}\label{Def1}  
Let $(S,\cdot)$ be a semigroup. 
\begin{itemize}
\item[(i)] A \emph{norm} on $S$ is a map $v: S \to \R_{\geq 0}$ such that
\begin{equation*}
v(x \cdot y) \leq v(x) + v(y)\ \text{for every}\ x,y\in S.
\end{equation*}
A  \emph{normed semigroup} is a semigroup provided with a norm.

If $S$ is a monoid, a \emph{monoid norm} on $S$ is a semigroup norm $v$ such that $v(1)=0$; in such a case $S$ is called \emph{normed monoid}.

\item[(ii)] A semigroup homomorphism $\f:(S,v)\to (S',v')$ between normed semigroups is \emph{contractive} if $$v'(\phi(x))\leq v(x)\ \text{for every}\ x\in S.$$
\end{itemize}
\end{Definition}

Let $\Se$ be the category of normed semigroups, which has as morphisms all contractive semigroup homomorphisms. In this paper, when we say that $S$ is a normed semigroup and $\phi:S\to S$ is an endomorphism, we will always mean that $\phi$ is a contractive semigroup endomorphism.
Moreover, let $\mathfrak M$ be the non-full subcategory of $\Se$ with objects all normed monoids, where the morphisms are all (necessarily contractive) monoid homomorphisms.

\medskip
We give now some other definitions.

\begin{Definition} 
A normed semigroup $(S,v)$ is:
\begin{itemize}
\item[(i)] \emph{bounded} if there exists $C\in \N_+$ such that $v(x) \leq C$ for all $x\in S$;
\item[(ii)]\emph{arithmetic} if for every $x\in S$ there exists a constant $C_x\in \N_+$ such that $v(x^n) \leq C_x\cdot  \log (n+1)$ for every $n\in\N$.
\end{itemize}
\end{Definition}

Obviously, bounded semigroups are arithmetic.

\begin{Example}\label{Fekete}
Consider the monoid $S = (\N, +)$.
\begin{itemize}
\item[(a)] Norms $v$ on $S$ correspond to  {subadditive sequences} $(a_n)_{n\in\N}$ in  $ \R_+$ 
(i.e., $a_{n + m}\leq a_n + a_m$) via 
$v \mapsto (v(n))_{n\in\N}$. Then $\lim_{n\to \infty} \frac{a_n}{n}= \inf_{n\in\N} \frac{a_n}{n}$ exists by Fekete Lemma \cite{Fek}.
\item[(b)]  Define $v: S \to \R_+$ by $v(x) =  \log (1+ x)$ for $x\in S$. Then 
$v$ is an arithmetic semigroup norm.  
\item[(c)] Define $v_1: S \to \R_+$ by $v_1(x) = \sqrt x$ for $x\in S$. Then $v_1$ is a  semigroup norm, but $(S, + , v_1)$ is not arithmetic.
\item[(d)] For $a\in \N$, $a>1$ let $v_a(n) = \sum_i b_i$, when $n= \sum_{i=0}^k b_ia^i$
and $0\leq b_i < a$ for all $i$. Then $v_a$ is an arithmetic norm on $S$ making the map $x\mapsto ax$ an endomorphism in $\Se$.  
\end{itemize}
\end{Example}

\subsection{Preordered semigroups and normed semilattices}\label{preorder-sec} 

A triple $(S,\cdot,\leq)$ is a \emph{preordered semigroup} if the semigroup $(S,\cdot)$ admits a preorder $\leq$ such that 
$$x\leq y\ \text{implies}\ x \cdot z \leq y \cdot z\ \text{and}\ z \cdot x \leq z \cdot y\ \text{for all}\ x,y,z \in S.$$
Write $x\sim y$ when $x\leq y$ and $y\leq x$ hold simultaneously.  Moreover, the \emph{positive cone} of $S$ is $$P_+(S)=\{a\in S:x\leq x \cdot a \ \text{and}\ x\leq a\cdot x\ \text{for every}\ x\in S\}.$$

A norm $v$ on the preordered semigroup $(S,\cdot,\leq)$ is \emph{monotone} if $x\leq y$ implies $v(x) \leq v(y)$ for every $x,y \in S$.  Clearly, $v(x) = v(y)$ whenever $x \sim y$ and the norm $v$ of $S$ is monotone.

Now we propose another notion of monotonicity for a semigroup norm which does not require the semigroup to be explicitly endowed with a preorder.
 
\begin{Definition}
Let $(S,v)$ be a normed semigroup. The norm $v$ is \emph{s-monotone} if $$\max\{v(x), v(y)\}\leq v(x \cdot y)\ \text{for every}\ x,y \in S.$$
\end{Definition}

This inequality may become a too stringent condition when $S$ is close to be a group; indeed, if $S$ is a group, then it implies that $v(S) = \{v(1)\}$, in particular $v$ is constant.

If $(S,+,v)$ is a commutative normed monoid, it admits a preorder $\leq^a$ defined for every $x,y\in S$ by
$x\leq^a y$ if and only if there exists $z\in S$ such that $x+z=y$. Then $(S,\cdot,\leq)$ is a {preordered semigroup} and
the norm $v$ is s-monotone if and only if $v$ is monotone with respect to $\leq^a$.

\medskip
The following connection between monotonicity and s-monotonicity is clear. 
 
\begin{Lemma} 
Let $S$ be a preordered semigroup. If $S=P_+(S)$, then every monotone norm of $S$ is also s-monotone.
\end{Lemma}   
 

A \emph{semilattice} is a commutative semigroup $(S,\vee)$ such that $x\vee x=x$ for every $x\in S$.
 
\begin{Example} 
\begin{itemize}
\item[(a)] Each lattice $(L, \vee, \wedge)$ gives rise to two semilattices, namely $(L, \vee)$ and $(L, \wedge)$.
\item[(b)] A filter $\mathcal F$ on a given set $X$ is a semilattice with respect to the intersection, with zero element the set $X$. 
\end{itemize}
\end{Example}
 
Let $\SL$ be the full subcategory of $\Se$ with objects all normed semilattices. 
 
Every normed semilattice $(L,\vee)$ is trivially arithmetic, moreover the canonical partial order defined by $$x\leq y\ \text{if and only if}\ x\vee y=y,$$ 
for every $x,y\in L$, makes $L$ also a partially ordered semigroup.

\medskip
Neither preordered semigroups nor normed semilattices are formally needed for the definition of the semigroup entropy. Nevertheless, they provide significant and natural examples, as well as useful tools in the proofs, to justify our attention to this topic.

\subsection{Entropy in $\Se$}\label{hs-sec}

For $(S,v)$ a normed semigroup $\phi:S\to S$ an endomorphism, $x\in S$ and $n\in\N_+$ consider the \emph{$n$-th $\phi$-trajectory of $x$}
$$T_n(\f,x) = x \cdot\f(x)\cdot\ldots \cdot\f^{n-1}(x)$$
and let $$c_n(\f,x) = v(T_n(\f,x)).$$ 
Note that $c_n(\phi,x) \leq n\cdot v(x)$. Hence the growth of the function $n \mapsto c_n(\f,x)$ is at most linear. 

\begin{Definition}
Let $S$ be a normed semigroup. An  endomorphism $\f:S\to S$ is said to have \emph{logarithmic growth}, if for every $x\in S$ there exists $C_x\in\N_+$ with $c_n(\f,x) \leq C_x\cdot \log (n+1)$ for all $n\in\N_+$.
\end{Definition}
 
Obviously, a normed semigroup $S$ is arithmetic if and only if $\id_{S}$ has logarithmic growth.
 
\medskip
The following theorem from \cite{DGV1} is fundamental in this context as it witnesses the existence of the semigroup entropy; so we give its proof also here for reader's convenience.

\begin{Theorem}\label{limit}
Let $S$ be a normed semigroup and $\f:S\to S$ an endomorphism. Then for every $x \in S$ the limit 
\begin{equation}\label{hs-eq}
h_{\Se}(\f,x):= \lim_{n\to\infty}\frac{c_n(\f,x)}{n}
\end{equation}
 exists and satisfies $h_{\Se}(\f,x)\leq v(x)$.  
\end{Theorem}
\begin{proof} 
The sequence $(c_n(\f,x))_{n\in\N_+}$ is subadditive. 
Indeed,  
\begin{align*}  
c_{n+m}(\f,x)&= v(x\cdot\f(x)\cdot\ldots\cdot\f^{n-1}(x)\cdot\f^{n}(x)\cdot\ldots\cdot\f^{n+m-1}(x))\\  
&=v((x\cdot\f(x)\cdot\ldots\cdot\f^{n-1}(x))\cdot\f^{n}(x\cdot\ldots\cdot\f^{m-1}(x))) \\ 
&\leq c_n(\phi,x)+v(\f^{n}(x\cdot\ldots\cdot\f^{m-1}(x))) \\  
&\leq c_n(\f,x)+v(x\cdot\ldots\cdot\f^{m-1}(x))=c_n(\f,x)+c_m(\f,x).
\end{align*} 
By Fekete Lemma (see Example \ref{Fekete} (a)),  the limit $\lim_{n\to\infty} \frac{c_n(\f,x)}{n}$ exists and coincides with $\inf_{n\in\N_+} \frac{c_n(\f,x)}{n}$.
Finally, $h_{\Se}(\f,x)\leq v(x)$ follows from $c_n(\phi,x) \leq n v(x)$ for every $n\in\N_+$.
\end{proof}    

\begin{Remark}
\begin{itemize}
\item[(a)] The proof of the existence of the limit defining $h_{\Se}(\f,x) $ exploits the property of the semigroup norm and also the condition on $\phi$ to be contractive. 
For an extended comment on what can be done in case the function $v: S \to \R_+$ fails to have that property see \S \ref{NewSec1}.
\item[(b)] With $S = (\N,+)$, $\phi = \id_\N$ and $x=1$ in Theorem \ref{limit} we obtain exactly item (a) of Example \ref{Fekete}. 
\end{itemize}
\end{Remark}

\begin{Definition}\label{SEofEndos}    
Let $S$ be a normed semigroup and $\f:S\to S$ an endomorphism. The \emph{semigroup entropy} of $\f$ is $$h_{\Se}(\f)=\sup_{x\in S}h_{\Se}(\f,x).$$
\end{Definition}

 
If an endomorphism $\phi:S\to S$ has logarithmic growth, then $h_{\Se}(\f) = 0$. In particular, $h_{\Se}(\id_S)=0$ if $S$ is arithmetic. 

Recall that an endomorphism $\phi:S\to S$ of a normed semigroup $S$ is \emph{locally quasi periodic} if for every $x\in S$ there exist $n,k\in\N$, $k>0$, such that $\phi^n(x)=\phi^{n+k}(x)$. If $S$ is a monoid and $\phi(1)=1$, then $\phi$ is \emph{locally nilpotent} if for every $x\in S$ there exists $n\in\N_+$ such that $\phi^n(x)=1$. 

\begin{Lemma}\label{locally}
Let $S$ be a normed semigroup and $\phi:S\to S$ an endomorphism. 
\begin{itemize}
\item[(a)]If $S$ is arithmetic and $\phi$ is locally periodic, then $h_\Se(\phi)=0$.
\item[(b)] If $S$ is a monoid and $\phi(1)=1$ and $\phi$ is locally nilpotent, then $h_\Se(\phi)=0$.
\end{itemize}
\end{Lemma}
\begin{proof}
(a) Let $x\in S$, and let $l,k\in\N_+$ be such that $\phi^l(x)=\phi^{l+k}(x)$. For every $m\in\N_+$ one has $$T_{l+mk}(\phi,x)=T_l(\phi,x)\cdot T_m(\id_S,y) = T_l(\phi,x)\cdot y^m,$$ where $y=\phi^l(T_k(\phi,x))$.
Since $S$ is arithmetic, there exists $C_x\in \N_+$ such that
\begin{equation*}\begin{split}
v(T_{l+mk}(\phi,x)) = v(T_l(\phi,x)\cdot y^m) \leq  \\ v(T_l(\phi,x)) + v( y^m) \leq  v(T_l (\phi,x)) + C_x\cdot\log (m+1),
\end{split}\end{equation*}
 so $\lim_{m\to \infty} \frac{v(T_{l+mk}(\phi,x))}{l+mk}=0$. Therefore we have found a subsequence of $(c_n(\phi,x))_{n\in\N_+}$ converging to $0$, hence also $h_\Se(\phi,x)=0$. Hence $h_\Se(\phi)=0$.

(b) For $x\in S$, there exists $n\in\N_+$ such that $\phi^n(x)=1$. Therefore $T_{n+k}(\phi,x)=T_n(\phi,x)$ for every $k\in\N$, hence $h_\Se(\phi,x)=0$.
\end{proof}

We discuss now a possible different notion of semigroup entropy. 
Let $(S,v)$ be a normed semigroup, $\phi:S\to S$ an endomorphism, $x\in S$ and $n\in\N_+$. One could define also the ``left'' $n$-th $\phi$-trajectory of $x$ as $$T_n^{\#}(\phi,x)=\phi^{n-1}(x)\cdot\ldots\cdot\phi(x)\cdot x,$$ changing the order of the factors with respect to the above definition.
With these trajectories it is possible to define another entropy letting $$h_\Se^{\#}(\phi,x)=\lim_{n\to\infty}\frac{v(T_n^{\#}(\phi,x))}{n},$$ and $$h_\Se^{\#}(\phi)=\sup\{h_\Se^{\#}(\phi,x):x\in S\}.$$
In the same way as above, one can see that the limit defining $h_\Se^{\#}(\phi,x)$ exists.

\medskip
Obviously $h_\Se^{\#}$ and $h_\Se$ coincide on the identity map and on commutative normed semigroups, but now we see that in general they do not take always the same values. Item (a) in the following example shows that it may occur the case that they do not coincide ``locally'', while they coincide ``globally''.
Moreover, modifying appropriately the norm in item (a), J. Spev\'ak found the example in item (b) for which $h_\Se^{\#}$ and $h_\Se$ do not coincide even ``globally''.

\begin{Example}
Let $X=\{x_n\}_{n\in\Z}$ be a faithfully enumerated countable set and let $S$ be the free semigroup generated by $X$. An element $w\in S$ is a word $w=x_{i_1}x_{i_2}\ldots x_{i_m}$ with $m\in\N_+$ and $i_j\in\Z$ for $j= 1,2, \ldots,  m$. In this case $m$ is called the {\em length} $\ell_X(w)$ of $w$, and a subword of $w$ is any $w'\in S$ of the form $w'=x_{i_k}x_{i_k+1}\ldots x_{i_l}$ with $1\le k\le l\le n$. 

Consider the automorphism $\phi:S\to S$ determined by $\phi(x_n)=x_{n+1}$ for every $n\in\Z$.

\begin{itemize}\label{ex-jan}
\item[(a)] Let $s(w)$ be the number of adjacent pairs $(i_k,i_{k+1})$ in $w$ such that $i_k<i_{k+1}$. The map $v:S\to\R_+$ defined by $v(w)=s(w)+1$ is a semigroup norm. 
Then $\phi:(S,v)\to (S,v)$ is an automorphism of normed semigroups.

It is straightforward to prove that, for $w=x_{i_1}x_{i_2}\ldots x_{i_m}\in S$,
\begin{itemize}
\item[(i)] $h_\Se^\#(\phi,w)=h_\Se(\phi,w)$ if and only if $i_1>i_m+1$;
\item[(ii)] $h_\Se^\#(\phi,w)=h_\Se(\phi,w)-1$ if and only if $i_m=i_1$ or $i_m=i_1-1$.
\end{itemize}
Moreover, 
\begin{itemize}
\item[(iii)]$h_\Se^\#(\phi)=h_\Se(\phi)=\infty$.
\end{itemize}
In particular, $h_\Se(\phi,x_0)=1$ while $h_\Se^\#(\phi,x_0)=0$.
%

\item[(b)] Define a semigroup norm $\nu: S\to \R_+$ as follows. For $w=x_{i_1}x_{i_2}\ldots x_{i_n}\in S$ consider its subword $w'=x_{i_k}x_{i_{k+1}}\ldots x_{i_l}$ with maximal length satisfying $i_{j+1}=i_j+1$ for every $j\in\Z$ with $k\le j\le l-1$ and let $\nu(w)=\ell_X(w')$.
Then $\phi:(S,\nu)\to (S,\nu)$ is an automorphism of normed semigroups.

It is possible to prove that, for $w\in S$,
\begin{enumerate}
\item[(i)] if $\ell_X(w)=1$, then $\nu(T_n(\phi,w))=n$ and $\nu(T^\#_n(\phi,w))=1$ for every $n\in\N_+$;
\item[(ii)] if $\ell_X(w)=k$ with $k>1$, then $\nu(T_n(\phi,w))< 2k$ and $\nu(T^\#_n(\phi,w))< 2k $ for every $n\in\N_+$.
\end{enumerate}
From (i) and (ii) and from the definitions we immediately obtain that
\begin{itemize}
\item[(iii)] $h_\mathfrak{S}(\phi)=1\neq 0=h^\#_\mathfrak{S}(\phi)$. 
\end{itemize}
\end{itemize}
\end{Example}

We list now the main basic properties of the semigroup entropy. For complete proofs and further details see 	\cite{DGV1}.

\begin{Lemma}[Monotonicity for factors] 
Let $S$, $T$ be normed semigroups and $\f: S \to S$, $\psi:T\to T$ endomorphisms. If $\alpha:S\to T$ is a surjective homomorphism such that $\alpha \circ \psi = \f \circ \alpha$, then $h_{\Se}(\f) \leq h_{\Se}(\psi)$. 
\end{Lemma}
\begin{proof} Fix $x\in S$ and find $y \in T$ with $x= \alpha(y)$. Then $c_n(x, \phi) \leq c_n(\psi, y)$.
Dividing by $n$ and taking the limit gives $h_{\Se}(\f,x) \leq h_{\Se}(\psi,y)$. So $h_{\Se}(\f,x)\leq h_{\Se}(\psi)$. When $x$ runs over $S$, we conclude that $h_{\Se}(\f) \leq h_{\Se}(\psi)$. 
\end{proof}    
 
\begin{Corollary}[Invariance under conjugation] 
Let $S$ be a normed semigroup and $\f: S \to S$ an endomorphism. If $\alpha:T\to S$ is an isomorphism, then $h_{\Se}(\f)=h_{\Se}(\alpha\circ\f\circ\alpha^{-1})$. 
\end{Corollary}

\begin{Lemma}[Invariance under inversion]\label{inversion}
Let $S$ be a normed semigroup and $\phi:S\to S$ an automorphism. Then $h_{\Se}(\f^{-1})=h_{\Se}(\f)$. 
\end{Lemma}
   
\begin{Theorem}[Logarithmic Law] 
Let $(S,v)$ be a normed semigroup and $\f:S\to S$ an endomorphism.  Then 
$$
h_{\Se}(\f^{k})\leq k\cdot h_{\Se}(\f)
$$ 
for every $k\in \N_+$. Furthermore, equality holds if $v$ is s-monotone. Moreover, if $\phi:S\to S$ is an automorphism, then $$h_{\Se}(\phi^k) = |k|\cdot h_{\Se}(\phi)$$ for all $k \in \Z\setminus\{0\}$.
\end{Theorem}
\begin{proof} 
Fix $k \in \N_+$, $x\in S$ and let $y= x\cdot\f(x)\cdot\ldots\cdot\f^{k-1}(x)$. Then 
\begin{align*}
h_\Se(\f^k)\geq h_\Se(\f^k, y)&=\lim_{n\to\infty} \frac{c_{n}(\f^k,y)}{n}=\lim_{n\to \infty} \frac{v (y\cdot \f^k(y)\cdot\ldots \cdot\f^{(n-1)k}(y)) }{n}=\\
&= k \cdot \lim_{n\to \infty} \frac{c_{nk}(\f,x)}{nk}=k\cdot h_\Se(\f,x).
\end{align*} 
This yields $h_\Se(\f^k)\geq k\cdot h_\Se(\f,x)$ for all $x\in S$, and consequently, $h_\Se(\f^k)\geq k\cdot h_\Se(\f)$. 
  
Suppose $v$ to be s-monotone, then
\begin{equation*}\begin{split}
h_\Se(\f,x)=\lim_{n\to \infty} \frac{v(x\cdot\f (x)\cdot\ldots\cdot\f^{nk-1}(x))}{n\cdot k} \geq \\
 \lim_{n\to\infty} \frac{ v(x\cdot\f^k(x)\cdot\ldots\cdot(\f^k)^{n-1}(x))}{n\cdot k}= \frac{h_\Se(\f^k,x)}{k}
\end{split}\end{equation*}
Hence, $k\cdot h_\Se(\f)\geq h_\Se(\f^k,x)$ for every $x\in S$. Therefore, $k\cdot h_\Se(\f)\geq h_\Se(\f^k)$. 

If $\phi$ is an automorphism and $k\in\Z\setminus\{0\}$, apply the previous part of the theorem and Lemma \ref{inversion}.
\end{proof}    

The next lemma shows that monotonicity is available not only under taking factors: 

\begin{Lemma}[Monotonicity for subsemigroups] 
Let $(S,v)$ be a normed semigroup and $\phi:S\to S$ an endomorphism. If $T$ is a $\phi$-invariant normed subsemigroup of $(S,v)$, then $h_{\Se}(\phi)\geq h_{\Se}(\phi\restriction_{T})$.  
Equality holds if $S$ is ordered, $v$ is monotone and $T$ is cofinal in $S$. 
\end{Lemma}
   
Note that $T$ is equipped with the induced norm $v\restriction_T$. The same applies to the subsemigroups $S_i$ in the next corollary: 

\begin{Corollary}[Continuity for direct limits] 
Let $(S,v)$ be a normed semigroup and $\phi:S\to S$ an endomorphism. If $\{S_i: i\in I\}$ is a directed family of 
$\phi$-invariant normed subsemigroup of $(S,v)$ with $ S =\varinjlim S_i$,  then $h_{\Se}(\phi)=\sup h_{\Se}(\phi\restriction_{S_i})$.
\end{Corollary}

We consider now products in $\Se$.
Let $\{(S_i,v_i):i\in I\}$ be a family of normed semigroups and let $S=\prod_{i \in I}S_i$ be their direct product in the category of semigroups. 

In case $I$ is finite, then $S$ becomes a normed semigroup with the $\max$-norm $v_{\prod}$, so $(S,v_{\prod})$ is the product of the family $\{S_i:i\in I\}$ in the category $\Se$; in such a case one has the following 

\begin{Theorem}[Weak Addition Theorem - products]\label{WAT}
Let $(S_i,v_i)$ be a normed semigroup and $\f_i:S_i\to S_i$ an endomorphism for $i=1,2$. Then the endomorphism $\f_1 \times \f_2$ of $ S _1 \times S_2$ has $h_\Se(\f_1 \times \f_2)= \max\{ h_\Se(\f_1),h_\Se(\f_2)\}$.
\end{Theorem}

If $I$ is infinite, $S$ need not carry a semigroup norm $v$ such that every projection $p_i: (S,v) \to (S_i,v_i)$ is a morphism in $\Se$. This is why the product of the family $\{(S_i,v_i):i\in I\}$ in $\Se$ is actually the subset $$S_{\mathrm{bnd}}=\{x=(x_i)_{i\in I}\in S: \sup_{i\in I}v_i(x_i)\in\R\}$$ of $S$ with the norm $v_{\prod}$ defined by $$v_{\prod}(x)=\sup_{i\in I}v_i(x_i)\ \text{for any}\ x=(x_i)_{i\in I}\in S_{\mathrm{bnd}}.$$ For further details in this direction see \cite{DGV1}.

\subsection{Entropy in $\mathfrak M$}

We collect here some additional properties of the semigroup entropy in the category $\mathfrak M$ of normed monoids where also coproducts are available. If $(S_i,v_i)$ is a normed monoid for every $i\in I$, the direct sum 
$$S= \bigoplus_{i\in I} S_i =\{(x_i)\in \prod_{i\in I}S_i: |\{i\in I: x_i \ne 1\}|<\infty\}$$ becomes a normed monoid with the norm $$v_\oplus(x) = \sum_{i\in I} v_i(x_i)\ \text{for any}\ x = (x_i)_{i\in I} \in S.$$ 
This definition makes sense since $v_i$ are monoid norms, so 
$v_i(1) = 0$. Hence, $(S,v_\oplus)$ becomes a coproduct of the family $\{(S_i,v_i):i\in I\}$ in $\mathfrak  M$. 

We consider now the case when $I$ is finite, so assume without loss of generality that $I=\{1,2\}$. In other words we have two normed monoids $(S_1,v_1)$ and $(S_2,v_2)$. The product and the coproduct have the same underlying monoid $S=S_1\times S_2$, but the norms $v_\oplus$ and $v_{\prod}$ in $S$ are different and give different values of the semigroup entropy $h_\Se$; indeed, compare Theorem \ref{WAT} and the following one.

\begin{Theorem}[Weak Addition Theorem - coproducts] 
Let $(S_i,v_i)$ be a normed monoid and $\f_i:S_i\to S_i$ an endomorphism for $i=1,2$. Then the endomorphism $\f_1 \oplus \f_2$ of $S _1 \oplus S_2$ has $h_\Se(\f_1 \oplus \f_2)= h_\Se(\f_1) + h_\Se(\f_2)$. 
\end{Theorem}
 
For a normed monoid $(M,v) \in \mathfrak  M$ let $B(M)= \bigoplus_\N M$, equipped with the above coproduct norm $v_\oplus(x) = \sum_{n\in\N} v(x_n)$ for any $x=(x_n)_{n\in\N}\in B(M)$. 
The \emph{right Bernoulli shift} is defined by $$\beta_M:B(M)\to B(M), \ \beta_M(x_0,\dots,x_n,\dots)=(1,x_0,\dots,x_n,\dots),$$
while the \emph{left Bernoulli shift} is $${}_M\beta:B(M)\to B(M),\ {}_M\beta(x_0,x_1,\dots,x_n,\dots)=(x_1,x_2, \dots,x_n,\dots).$$

\begin{Theorem}[Bernoulli normalization]  
Let $(M,v)$ be a normed monoid. Then:
\begin{itemize}
\item[(a)] $h_\Se (\beta_M)=\sup_{x\in M}v(x)$; 
\item[(b)] $h_\Se({}_M\beta) = 0$.
\end{itemize}
\end{Theorem}
\begin{proof}     
(a) For $x\in M$ consider $\underline{x}=(x_n)_{n\in\N}\in B(M)$ such that $x_0=x$ and $x_n=1$ for every $n\in\N_+$.  
Then $v_\oplus(T_n(\beta_M,\underline{x}))=n\cdot v(x)$, so $h_\Se(\beta_M,\underline{x})=v(x)$.   Hence $h_\Se(\beta_M)\geq \sup_{x\in M}v(x)$.  
Let now $\underline{x}=(x_n)_{n\in\N}\in B(M)$ and let $k\in\N$ be the greatest index such that $x_k\neq 1$; then  
\begin{equation*}\begin{split}
v_\oplus(T_n(\beta_M,\underline{x}))= \sum_{i=0}^{k+n} v(T_n(\beta_M,\underline{x})_i)\leq\\
\sum_{i=0}^{k-1} v(x_0\cdot\ldots\cdot x_i) + (n-k)\cdot v(x_1\cdot\ldots\cdot x_k)+\sum_{i=1}^{k} v(x_i\cdot\ldots\cdot x_k).
\end{split}\end{equation*}
Since the first and the last summand do not depend on $n$, after dividing by $n$ and letting $n$ converge to infinity we obtain  
$$h_\Se(\beta_M,\underline{x})=\lim_{n\to \infty} \frac{v_\oplus(T_n(\beta_M,\underline{x}))}{n}\leq v(x_1\cdot\ldots\cdot x_k)\leq \sup_{x\in M}v(x).$$ 

(b) Note that ${}_M\beta$ is locally nilpotent and apply Lemma \ref{locally}.
\end{proof}

\subsection{Semigroup entropy of an element and pseudonormed semigroups}\label{NewSec1}

One can notice a certain asymmetry in Definition \ref{SEofEndos}. 
 Indeed, for $S$ a normed semigroup, the local semigroup entropy defined in \eqref{hs-eq} is a two variable function 
$$h_{\Se}: \End(S) \times S \to  \R_+.$$ 
Taking $h_{\Se}(\f)=\sup_{x\in S}h_{\Se}(\f,x)$ for an endomorphism $\f\in\End(S)$, we obtained the notion of semigroup entropy of $\f$. 
But one can obviously exchange the roles of $\f$ and $x$ and obtain the possibility to discuss the entropy of an element $x\in S$. This can be 
done in two ways. Indeed, in Remark \ref{Asymm} we consider what seems the natural counterpart of $h_{\Se}(\f)$, while
here we discuss a particular case that could appear to be almost trivial, but actually this is not the case, as it permits to give a uniform approach to some entropies which are not defined by using trajectories. 
So, by taking $\phi=\id_S$ in \eqref{hs-eq}, we obtain a map $h_\Se^0:S\to\R_+$:

\begin{Definition}
Let $S$ be a normed semigroup and $x\in S$.
The \emph{semigroup entropy} of $x$ is 
$$
h_{\Se}^0(x):=h_{\Se}(\id_S,x) = \lim_{n\to\infty} \frac{v(x^n)}{n}.
$$
\end{Definition}

We shall see now that the notion of semigroup entropy of an element is supported by many examples. On the other hand, since some of the examples given below cannot be covered by our scheme, we propose first a slight extension that covers those examples as well. 

Let  ${\mathfrak S}^*$  be the category having as objects of all pairs $(S,v)$,  where $S$ is a semigroup and $v:S \to \R_+$  is an \emph{arbitrary} map. 
A morphism in the category ${\mathfrak S}^*$ is a semigroup homomorphism $\f: (S,v) \to (S',v')$  that is contracting with respect to the pair $v,v'$, i.e., $v'(\f(x)) \leq v(x)$ for every $x\in S$. 
Note that our starting category ${\mathfrak S}$ is simply a full subcategory of ${\mathfrak S}^*$, having as objects those pairs $(S,v)$ such that $v$ satisfies (i) from Definition \ref{Def1}. These pairs were called normed semigroups and $v$ was called a semigroup norm. For the sake of convenience and in order to keep close to the current 
terminology, let us call the function $v$ in the larger category ${\mathfrak S}^*$ a \emph{semigroup pseudonorm} (although, we are imposing no condition on $v$ whatsoever). 

So, in this setting, one can define a local semigroup entropy $h_{\Se^*}: \End(S) \times S \to  \R_+$ following the pattern of \eqref{hs-eq}, replacing the limit by 
$$h_{\Se^*} (\phi,x)=\limsup_{n\to \infty}\frac{v(T_n(\phi,x))}{n}.$$
In particular, 
$$h_{\Se^*}^0(x)=\limsup_{n\to \infty}\frac{v(x^n)}{n}.$$
Let us note that in order to have the last $\limsup$ a limit, one does not need $(S,v)$ to be in $\Se$, but it suffices to have the semigroup norm condition (i) from Definition \ref{Def1} fulfilled only for products of powers of the same element. 

\medskip
We consider here three different entropies, respectively from \cite{MMS}, \cite{FFK} and \cite{Silv}, that can be described in terms of $h_\Se^0$ or its generalized version $h_{\Se^*}^0$. We do not go into the details, but we give the idea how to capture them using the notion of semigroup entropy of an element of the semigroup of all endomorphisms of a given object equipped with an appropriate semigroup (pseudo)norm. 

\begin{itemize}
\item[(a)] Following \cite{MMS}, let $R$ be a Noetherian local ring and $\phi:R\to R$ an endomorphism of finite length; moreover, $\lambda(\phi)$ is the length of $\phi$, which is a real number $\geq 1$. In this setting the entropy of $\phi$ is defined by $$h_\lambda(\phi)=\lim_{n\to \infty}\frac{\log\lambda(\phi^n)}{n}$$ and it is proved that this limit exists.

Then the set $S=\End_{\mathrm{fl}}(R)$ of all finite-length endomorphisms of $R$ is a semigroup and $\log\lambda(-)$ is a semigroup norm on $S$. 
For every $\phi\in S$, we have 
$$
h_\lambda(\phi)=h_\Se(\id_S,\phi)=h_{\Se}^0(\f).
$$
In other words, $h_\lambda(\f)$ is nothing else but the semigroup entropy of the element $\f$ of the normed semigroup $S=\End_{\mathrm{fl}}(R)$. 
   
\item[(b)] We recall now the entropy considered in \cite{Silv}, which was already introduced in \cite{BV}. Let $t\in\N_+$ and $\varphi:\mathbb P^t\to\mathbb P^t$ be a dominant rational map of degree $d$. Then the entropy of $\varphi$ is defined as the logarithm of the dynamical degree, that is
$$
h_\delta (\varphi)=\log \delta_\phi=\limsup_{n\to \infty}\frac{\log\deg(\varphi^n)}{n}.
$$
Consider the semigroup $S$ of all dominant rational maps of $\mathbb P^n$ and the function $\log\deg(-)$. In general this is only a 
semigroup pseudonorm on $S$ and $$h_{\Se^*}^0(\varphi)=h_\delta(\varphi).$$ 
Note that $\log\deg(-)$ is a semigroup norm when $\varphi$ is an endomorphism of the variety $\mathbb P^t$.

\item[(c)] We consider now the growth rate for endomorphisms introduced in \cite{Bowen} and recently studied in \cite{FFK}. Let $G$ be a finitely generated group, $X$ a finite symmetric set of generators of $G$, and $\varphi:G\to G$ an endomorphism. For $g\in G$, denote by $\ell_X(g)$ the length of $g$ with respect to the alphabet $X$. The growth rate of $\varphi$ with respect to $x\in X$ is
$$\log GR(\varphi,x)=\lim_{n\to \infty}\frac{\log \ell_X(\varphi^n(x))}{n}$$
(and the growth rate of $\varphi$ is $\log GR(\varphi)=\sup_{x\in X} \log GR(\varphi,x)$).

Consider $S=\End(G)$ and, fixed $x\in X$, the map $\log GR(-,x)$. As in item (b) this is only a semigroup pseudonorm on $S$. Nevertheless, also in this case the semigroup entropy $$\log GR(\varphi,x)=h_{\Se^*}^0(\varphi).$$
\end{itemize}

\begin{Remark}\label{Asymm}
For a normed semigroup $S$, let $h_{\Se}: \End(S) \times S \to  \R_+$ be the local semigroup entropy defined in \eqref{hs-eq}.   
 Exchanging the roles of $\f\in \End(S)$ and $x\in S$,  define the \emph{global semigroup entropy} of an element $x\in S$ by 
 $$
 h_{\Se}(x)=\sup_{\f \in \End(S)}h_{\Se}(\f,x).
 $$
Obviously, $h_{\Se}^0(x) \leq h_{\Se}(x)$ for every $x\in S$. 
\end{Remark}

\section{Obtaining known entropies}\label{known-sec}

\subsection{The general scheme}

Let $\mathfrak X$ be a category and let $F:\mathfrak X\to \Se$ be a functor. Define the entropy $$h_{F}:\mathfrak X\to \mathbb R_+$$ on the category $\mathfrak X$ by $$h_{F}(\phi)=h_{\Se}(F(\phi)),$$
for any endomorphism $\phi: X \to X$ in $\mathfrak X$. Recall that with some abuse of notation we write $h_{F}:\mathfrak X\to \mathbb R_+$ in place of $h_{F}:\mathrm{Flow}_\mathfrak X\to \mathbb R_+$ for simplicity.
   
\medskip
Since the functor $F$ preserves commutative squares and isomorphisms, the entropy $h_{F}$ has the following properties, that automatically follow from the previously listed properties of the semigroup entropy $h_\Se$. For the details and for properties that need a further discussion see \cite{DGV1}.

Let $X$, $Y$ be objects of $\mathfrak X$ and $\phi:X\to X$, $\psi:Y\to Y$ endomorphism in $\mathfrak X$.
\begin{itemize}
\item[(a)][Invariance under conjugation]
If $\alpha:X\to Y$ is an isomorphism in $\mathfrak X$, then $h_{F}(\phi)=h_{F}(\alpha\circ\phi\circ\alpha^{-1})$.
\item[(b)][Invariance under inversion]
If $\phi:X\to X$ is an automorphism in $\mathfrak X$, then $h_{F}(\f^{-1})=h_{F}(\f)$.
\item[(c)][Logaritmic Law] If the norm of $F(X)$ is $s$-monotone, then $h_{F}(\f^{k})=k\cdot h_{F}(\f)$ for all $k\in \N_+$. 
\end{itemize}
Other properties of $h_{F}$ depend on properties of the functor $F$.
\begin{itemize}
\item[(d)][Monotonicity for invariant subobjects] 
If $F$ sends subobject embeddings in $\mathfrak X$ to embeddings in $\Se$ or to surjective maps in $\Se$, then, if $Y$ is a $\f$-invariant subobject of $X$, we have $h_{F}(\f\restriction_Y)\leq h_{F}(\f)$.
\item[(e)][Monotonicity for factors] If $F$ sends factors in $\mathfrak X$ to surjective maps in $\Se$ or to embeddings in $\Se$, then, if $\alpha:T\to S$ is an epimorphism in $\mathfrak X$ such that $\alpha \circ \psi = \phi \circ \alpha$, then $h_F(\phi) \leq h_F(\psi)$.
\item[(f)][Continuity for direct limits] If $F$ is covariant and sends direct limits to direct limits, then $h_F(\phi)=\sup_{i\in I} h_F(\phi\restriction_{X_i})$ whenever $X=\varinjlim X_i$ and $X_i$ is a $\phi$-invariant subobject of $X$ for every $i\in I$.
\item[(g)][Continuity for inverse limits] If $F$ is contravariant and sends inverse limits to direct limits, then $h_F(\phi)=\sup_{i\in I} h_F(\overline\phi_i)$ whenever $X=\varprojlim X_i$ and $(X_i,\phi_i)$ is a factor of $(X,\phi)$ for every $i\in I$.
\end{itemize}

In the following subsections we describe how the known entropies can be obtained from this general scheme. For all the details we refer to \cite{DGV1}

\subsection{Set-theoretic entropy}

In this section we consider the category $\mathbf{Set}$ of sets and maps and its (non-full) subcategory $\mathbf{Set}_\mathrm{fin}$ having as morphisms all the finitely many-to-one maps. We construct a functor $\mathfrak{atr}:\mathbf{Set}\to\Se$ and a functor $\mathfrak{str}: \mathbf{Set}_\mathrm{fin} \to \Se$, which give the set-theoretic entropy $\mathfrak h$ and the covariant set-theoretic entropy $\mathfrak h^*$, introduced in \cite{AZD} and \cite{DG-islam} respectively. We also recall that they are related to invariants for self-maps of sets introduced in \cite{G0} and \cite{AADGH} respectively.

\medskip
A natural semilattice with zero, arising from a set $X$, is the family $({\mathcal S}(X),\cup)$ of all finite subsets of $X$ with neutral element $\emptyset$. Moreover the map defined by $v(A) = |A|$ for every $A\in\mathcal S(X)$ is an s-monotone norm. So let $\mathfrak{atr}(X)=(\mathcal S(X),\cup,v)$.
Consider now a map $\lambda:X\to Y$ between sets and define $\mathfrak{atr}(\lambda):\mathcal S(X)\to \mathcal S(Y)$ by $A\mapsto \lambda(A)$ for every $A\in\mathcal S(X)$.
This defines a covariant functor $$\mathfrak{atr}: \mathbf{Set} \to \Se$$
such that
$$h_{\mathfrak{atr}}=\mathfrak h.$$

Consider now a finite-to-one map $\lambda:X\to Y$. As above let $\mathfrak{str}(X)=(\mathcal S(X),\cup,v)$, while $\mathfrak{str}(\lambda):\mathfrak{str}(Y)\to\mathfrak{str}(X)$ is given by $A \mapsto\lambda^{-1}(A)$ for every $A\in\mathcal S(Y)$. This defines a contravariant functor
$$
\mathfrak{str}: \mathbf{Set}_\mathrm{fin}\to\Se
$$
such that
$$
h_{\mathfrak{str}}=\mathfrak h^*.
$$
 
%

\subsection{Topological entropy for compact spaces}

In this subsection we consider in the general scheme the topological entropy $h_{top}$ introduced in \cite{AKM} for continuous self-maps of compact spaces. So we specify the general scheme for the category $\mathfrak X=\mathbf{CTop}$ of compact spaces and continuous maps, constructing the functor $\cov:\mathbf{CTop}\to\Se$.

\medskip
For a topological space $X$ let $\cov(X)$ be the family of all open covers $\mathcal U$ of $X$, where it is allowed $\emptyset\in\mathcal U$. For $\U, \V\in \cov(X)$ let $\U \vee \V=\{U\cap V: U\in \U, V\in \V\}\in \cov(X)$.
One can easily prove commutativity and associativity of $\vee$; moreover, let $\mathcal E=\{X\}$ denote the trivial cover. 
Then 
\begin{center}
$(\cov(X), \vee, \mathcal E)$ is a commutative monoid.
\end{center}

For a topological space $X$, one has a natural preorder $\U \prec\V $ on $\cov(X)$; indeed, $\V$ \emph{refines} $\U$  if for every $V \in \V$ there exists  $U\in \U$ such that $V\subseteq U$. Note that this preorder has bottom element $\mathcal E$, and that it is not an order. 
In general,  $\U \vee \U \ne \U $, yet  $\U \vee \U \sim \U $, and more generally  
\begin{equation}\label{vee} 
\U \vee \U  \vee \ldots \vee \U \sim \U. 
\end{equation}

For $X$, $Y$ topological spaces, a continuous map $\f:X\to Y$ and $\U\in \cov (Y)$, let $\f^{-1}(\U)=\{\f^{-1}(U): U\in \U\}$. Then, as $\f^{-1}(\U \vee \V)= \f^{-1}(\U)\vee \f^{-1}(\V)$, we have that 
$\cov (\f): \cov (Y)\to \cov (X)$, defined by $\U \mapsto \f^{-1}(\U)$, is a semigroup homomorphism.
This defines a contravariant functor $\mathfrak{cov}$ from the category of all topological spaces to the category of commutative semigroups. 
 
\medskip
To get a semigroup norm on $\cov(X)$ we restrict this functor to the subcategory $\CT$ of compact spaces. For a compact space $X$ and $\U\in \cov(X)$, let $$M(\U)=\min\{|\V|: \V\mbox{ a finite subcover of }\U\}\ \text{and}\ v(\U)=\log M(\U).$$ Now \eqref{vee} gives $v(\U \vee \U  \vee \ldots \vee \U) = v(\U)$, so   
\begin{center}
$(\mathfrak{cov}(X), \vee, v)$ is an arithmetic normed semigroup.
\end{center}

For every continuous map $\f:X\to Y$ of compact spaces and $\W\in \cov(Y)$, the inequality  $v(\f^{-1}(\W))\leq v(\W)$ holds.
Consequently
\begin{center}
$\cov(\f): \cov (Y)\to \cov (X)$, defined by $\W\mapsto\phi^{-1}(\W)$, is a morphism in $\Se$.
\end{center}
Therefore the assignments $X \mapsto \cov(X)$ and $\phi\mapsto\cov(\phi)$ define a contravariant functor $$\cov:\mathbf{CTop}\to \Se.$$
Moreover,
$$h_{\mathfrak{cov}}=h_{top}.$$

Since the functor $\mathfrak{cov}$ takes factors in $\CT$ to embeddings in $\Se$, embeddings in $\CT$ to surjective morphisms in $\Se$, and inverse limits in $\CT$ to direct limits in $\Se$, we have automatically that the topological entropy $h_{top}$ is monotone for factors and restrictions to invariant subspaces, continuous for inverse limits, is invariant under conjugation and inversion, and satisfies the Logarithmic Law.

\subsection{Measure entropy}

In this subsection we consider the category $\MS$ of probability measure spaces $(X, \mathfrak B, \mu)$ and measure preserving maps, constructing a functor $\mathfrak{mes}:\MS\to\Se$ in order to obtain from the general scheme the measure entropy $h_{mes}$ from \cite{K} and \cite{Sinai}.

\medskip
For a measure space $(X,\mathfrak{B},\mu)$ let $\mathfrak{P}(X)$ be the family of all measurable partitions $\xi=\{A_1,A_2,\ldots,A_k\}$ of $X$. 
For $\xi, \eta\in \mathfrak{P}(X)$ let $\xi \vee \eta=\{U\cap V: U\in \xi, V\in \eta\}$. As $\xi \vee \xi = \xi$, with zero the cover $\xi_0=\{X\}$,
\begin{center}
$(\mathfrak{P}(X),\vee)$ is a semilattice with $0$.
\end{center}
Moreover, for $\xi=\{A_1,A_2,\ldots,A_k\}\in \mathfrak{P}(X)$  the \emph{entropy} of $\xi$ is given by Boltzmann's Formula  
$$
v(\xi)=-\sum_{i=1}^k \mu(A_k)\log \mu(A_k).
$$ 
This is a monotone semigroup norm making $\mathfrak{P}(X)$ a normed semilattice and a normed monoid.
 
Consider now a measure preserving map $T:X\to Y$. For a cover $\xi=\{A_i\}_{i=1}^k\in \mathfrak{P}(Y)$ let $T^{-1}(\xi)=\{T^{-1}(A_i)\}_{i=1}^k$. 
Since $T$ is measure preserving, one has $T^{-1}(\xi)\in \mathfrak{P}(X)$ and $\mu (T^{-1}(A_i)) = \mu(A_i)$ for all $i=1,\ldots,k$. Hence, $v(T^{-1}(\xi)) = v(\xi)$ and so
\begin{center}
$\mathfrak{mes}(T):\mathfrak{P}(Y)\to\mathfrak{P}(X)$, defined by $\xi\mapsto T^{-1}(\xi)$, is a morphism in $\SL$.
\end{center}
Therefore
the assignments $X \mapsto\mathfrak{P}(X)$ and $T\mapsto\mathfrak{mes}(T)$ define a contravariant functor $$\mathfrak{mes}:\MS\to\SL.$$
Moreover,
$$h_{\mathfrak{mes}}=h_{mes}.$$

The functor $\mathfrak{mes}:\MS\to\SL$ is covariant and sends embeddings in $\MS$ to surjective morphisms in $\SL$ and sends surjective maps in $\MS$ to embeddings in $\SL$. Hence, similarly to $h_{top}$, also the measure-theoretic entropy $h_{mes}$  is monotone for factors and restrictions to invariant subspaces, continuous for inverse limits, is invariant under conjugation and inversion, satisfies the Logarithmic Law and the Weak Addition Theorem.

\medskip
In the next remark we briefly discuss the connection between measure entropy and topological entropy.

\begin{Remark}
\begin{itemize}
\item[(a)] If $X$ is a compact metric space and $\phi: X \to X$ is a continuous surjective self-map, 
by Krylov-Bogolioubov Theorem \cite{BK} there exist some $\phi$-invariant Borel probability measures $\mu$ on $X$ (i.e., making $\phi:(X,\mu) \to (X,\mu)$ measure preserving). Denote by $h_\mu$ the measure entropy with respect to the measure $\mu$.
The inequality $h_{\mu}(\phi)\leq h_{top}(\phi)$ for every $\mu \in M(X,\phi)$ is due to Goodwyn \cite{Goo}. Moreover the \emph{variational principle} (see \cite[Theorem 8.6]{Wa}) holds true:
$$h_{top}(\phi)=  \sup \{h_{\mu}(\phi):  \mu\ \text{$\phi$-invariant measure on $X$}\}.$$
\item[(b)] In the computation of the topological entropy it is possible to reduce to surjective continuous self-maps of compact spaces. Indeed, for a compact space $X$ and a continuous self-map $\phi:X\to X$, the set $E_\phi(X)=\bigcap_{n\in\N}\phi^n(X)$ is closed and $\phi$-invariant, the map $\phi\restriction_{E_\phi(X)}:E_\phi(X)\to E_\phi(X)$ is surjective and $h_{top}(\phi)=h_{top}(\phi\restriction_{E_\phi(X)})$ (see \cite{Wa}). 
\item[(c)] In the case of a compact group $K$ and a continuous surjective endomorphism $\phi:K\to K$, the group $K$ has its unique Haar measure and so $\phi$ is measure preserving as noted by Halmos \cite{Halmos}.
In particular both $h_{top}$ and $h_{mes}$ are available for surjective continuous endomorphisms of compact groups and they coincide as proved in the general case by Stoyanov \cite{S}.

 In other terms, denote by $\mathbf{CGrp}$ the category of all compact groups and continuous homomorphisms, and by $\textbf{CGrp}_e$ the non-full subcategory of $\textbf{CGrp}$, having as morphisms all epimorphisms in $\textbf{CGrp}$. So in the following diagram we consider the forgetful functor $V: \textbf{CGrp}_e\to \textbf{Mes}$, while $i$ is the inclusion of $\textbf{CGrp}_e$ in $\textbf{CGrp}$ as a non-full subcategory and $U:\mathbf{CGrp}\to \mathbf{Top}$ is the forgetful functor:
\begin{equation*}
\xymatrix{
\textbf{CGrp}_e\ar[r]^{i}\ar[d]^V & \textbf{CGrp}\ar[r]^{U}&  \textbf{Top} \\
\textbf{Mes}
}
\end{equation*}
For a surjective endomorphism $\phi$ of the compact group $K$, we have then $h_{mes}(V(\phi))=h_{top}(U(\phi))$.
%
\end{itemize}
\end{Remark}

\subsection{Algebraic entropy}

Here we consider the category $\mathbf{Grp}$ of all groups and their homomorphisms and its subcategory $\mathbf{AbGrp}$ of all abelian groups. We construct two functors $\mathfrak{sub}:\mathbf{AbGrp}\to\SL$ and $\mathfrak{pet}:\mathbf{Grp}\to\Se$ that permits to find from the general scheme the two algebraic entropies $\ent$ and $h_{alg}$. For more details on these entropies see the next section.

\medskip
Let $G$ be an abelian group and let $(\sF(G),\cdot)$ be the semilattice consisting of all finite subgroups of $G$. Letting $v(F) = \log|F|$ for every $F \in \sF(G)$, then
\begin{center}
$(\sF(G),\cdot,v)$ is a normed semilattice
\end{center}
and the norm $v$ is monotone.

For every group homomorphism $\f: G \to H$, 
\begin{center}
the map $\sF(\f): \sF(G) \to \sF(H)$, defined by $F\mapsto \f(F)$, is a morphism in $\SL$.
\end{center}
Therefore
the assignments $G\mapsto \sF(G)$ and $\phi\mapsto \sF(\phi)$ define a covariant functor $$\mathfrak{sub}: \AG \to \SL.$$
Moreover $$h_{\mathfrak{sub}}=\ent.$$

Since the functor $\mathfrak{sub}$ takes factors in $\mathbf{AbGrp}$ to surjective morphisms in $\Se$, embeddings in $\mathbf{AbGrp}$ to embeddings in $\Se$, and direct limits in $\mathbf{AbGrp}$ to direct limits in $\Se$, we have automatically that the algebraic entropy $\ent$ is monotone for factors and restrictions to invariant subspaces, continuous for direct limits, invariant under conjugation and inversion, satisfies the Logarithmic Law.
 
\medskip
For a group $G$ let $\sM(G)$ be the family of all finite non-empty subsets of $G$. 
Then $\sM(G)$ with the operation induced by the multiplication of $G$ is a monoid with neutral element $\{1\}$. 
Moreover, letting $v(F) = \log |F|$ for every $F \in \sM(G)$ makes $\sM(G)$ a normed semigroup.
For an abelian group $G$ the monoid $\sM(G)$ is arithmetic since for any $F\in \sM(G)$ the sum of $n$ summands satisfies $|F + \ldots + F|\leq (n+1)^{|F|}$. 
Moreover, $(\sM(G),\subseteq)$ is an ordered semigroup and the norm $v$ is $s$-monotone. 

For every group homomorphism $\f:G \to H$, 
\begin{center}
the map $\sM(\f): \sM(G) \to \sM(H)$, defined by $F\mapsto \f(F)$, is a morphism in $\Se$.
\end{center}
Consequently the assignments $G \mapsto (\sM(G),v)$ and $\phi\mapsto \sM(\phi)$ give a covariant functor $$\mathfrak{pet}:\mathbf{Grp}\to \Se.$$
Hence $$h_{\mathfrak{pet}}=h_{alg}.$$
Note that the functor $\mathfrak{sub}$ is a subfunctor of $\mathfrak{pet}:  \AG  \to\Se$ as $ \sF(G) \subseteq  \sH(G)$ for every abelian group $G$. 
%

\medskip
As for the algebraic entropy $\ent$, since the functor $\mathfrak{pet}$ takes factors in $\mathbf{Grp}$ to surjective morphisms in $\Se$, embeddings in $\mathbf{Grp}$ to embeddings in $\Se$, and direct limits in $\mathbf{Grp}$ to direct limits in $\Se$, we have automatically that the algebraic entropy $h_{alg}$ is monotone for factors and restrictions to invariant subspaces, continuous for direct limits, invariant under conjugation and inversion, satisfies the Logarithmic Law.

\subsection{$h_{top}$ and $h_{alg}$ in locally compact groups}\label{NewSec2}

As mentioned above, Bowen introduced topological entropy for uniformly continuous self-maps of metric spaces in \cite{B}. 
His approach turned out to be especially efficient in the case of locally compact spaces provided with some Borel measure with good invariance
properties, in particular for {continuous endomorphisms of locally compact groups provided with their Haar measure}. 
Later Hood in \cite{hood} extended Bowen's definition to uniformly continuous self-maps of arbitrary uniform spaces and in particular to 
continuous endomorphisms of (not necessarily metrizable)  locally compact groups.

On the other hand, Virili \cite{V} extended the notion of algebraic entropy to continuous endomorphisms of locally compact abelian groups, inspired by Bowen's definition of topological entropy (based on the use of Haar measure). As mentioned in \cite{DG-islam}, his definition can be extended to continuous endomorphisms of arbitrary locally compact groups. 

\medskip
Our aim here is to show that both entropies can be obtained from our general scheme in the case of measure preserving topological automorphisms of locally compact groups. 
To this end we recall first the definitions of  $h_{top}$ and $h_{alg}$ in locally compact groups.  Let $G$ be a locally compact group, let $\mathcal C(G)$ be the family of all compact neighborhoods of $1$ and $\mu$ be a right Haar measure on $G$. 
For a continuous endomorphism $\phi: G \to G$, $U\in\mathcal C(G)$ and a positive integer $n$, the $n$-th cotrajectory $C_n(\phi,U)=U\cap \phi^{-1}(U)\cap\ldots\cap\phi^{-n+1}(U)$ is still in $\mathcal C(G)$. The topological entropy $h_{top}$ is intended to measure the rate of decay of the $n$-th cotrajectory $C_n(\phi,U)$.
So let
\begin{equation}
H_{top}(\phi,U)=\limsup _{n\to \infty} - \frac{\log \mu (C_n(\phi,U))}{n},
\end{equation}
which does not depend on the choice of the Haar measure $\mu$.
The \emph{topological entropy} of $\phi$ is 
$$
h_{top}(\phi)=\sup\{H_{top}(\phi,U):U\in\mathcal C(G)\}.
$$
If $G$ is discrete, then $\mathcal C(G)$ is the family of all finite subsets of $G$ containing $1$, and $\mu(A) = |A|$ for subsets $A$ of $G$. So $H_{top}(\phi,U)= 0$ for every $U \in \mathcal C(G)$, hence $h_{top}(\phi)=0$. 

To define the algebraic entropy of $\phi$ with respect to $U\in\mathcal C(G)$ one uses the {$n$-th $\phi$-trajectory} $T_n(\phi,U)=U\cdot \phi(U)\cdot \ldots\cdot \phi^{n-1}(U)$ of $U$, that still belongs to $\mathcal C(G)$. 
It turns out that the value 
\begin{equation}\label{**}
H_{alg}(\phi,U)=\limsup_{n\to \infty} \frac{\log \mu (T_n(\phi,U))}{n}
\end{equation}
does not depend on the choice of $\mu$.  The \emph{algebraic entropy} of $\phi$ is 
$$
h_{alg}(\phi)=\sup\{H_{alg}(\phi,U):U\in\mathcal C(G)\}.
$$
The term ``algebraic'' is motivated by the fact that the definition of $T_n(\phi,U)$ (unlike $C_n(\phi,U)$) makes use of the group operation.

As we saw above \eqref{**} is a limit when $G$ is discrete. Moreover, if $G$ is compact, then $h_{alg}(\phi)=H_{alg}(\phi,G)=0$.

\medskip
In the sequel, $G$ will be a locally compact group. We fix also a measure preserving topological automorphism $\phi: G \to G$.  

To obtain the entropy $h_{top}(\phi)$ via semigroup entropy fix some $V\in \mathcal C(G)$ with $\mu(V)\leq 1$. Then consider the subset 
$$
\mathcal C_0(G)=\{U\in \mathcal C(G): U \subseteq V\}.
$$ 
Obviously, $\mathcal C_0(G)$ is a monoid with respect to intersection, having as neutral element $V$. To obtain a pseudonorm $v$ on $\mathcal C_0(G)$ let $v(U) = - \log \mu (U)$ for any $U \in \mathcal C_0(G)$. 
Then $\phi$ defines a semigroup isomorphism $\phi^\#: \mathcal C_0(G)\to \mathcal C_0(G)$ by $\phi^\#(U) = \phi^{-1}(U)$ for any 
$U\in \mathcal C_0(G)$. It is easy to see that $\phi^\#: \mathcal C_0(G)\to \mathcal C_0(G)$ is a
an automorphism in $\Se^*$ and  the semigroup entropy $h_{\Se^*}(\phi^\#)$ coincides with  $h_{top}(\phi)$  since 
$H_{top}(\phi,U) \leq H_{top}(\phi,U')$ whenever $U \supseteq U'$. 

To obtain the entropy $h_{alg}(\phi)$ via semigroup entropy fix some $W\in \mathcal C(G)$ with 
$\mu(W)\geq 1$. Then consider the subset 
$$
\mathcal C_1(G)=\{U\in \mathcal C(G): U \supseteq W\}
$$ 
of the set $\mathcal C(G)$. Note that for $U_1, U_2 \in \mathcal C_1(G)$ also $U_1U_2 \in \mathcal C_1(G)$. 
Thus $\mathcal C_1(G)$ is a semigroup. To define a pseudonorm $v$ on $\mathcal C_1(G)$ let 
$v(U) = \log \mu (U)$ for any $U \in \mathcal C_1(G)$. Then $\phi$ defines a semigroup isomorphism $\phi_\#: \mathcal C_1(G)\to \mathcal C_1(G)$ by $\phi_\#(U) = \phi(U)$ for any $U\in \mathcal C_1(G)$. It is easy to see that $\phi_\#: \mathcal C_1(G)\to \mathcal C_1(G)$ is a morphism in $\Se^*$ and 
  the semigroup entropy $h_{\Se^*}(\phi_\#)$ coincides with  $h_{alg}(\phi)$, 
  since $ \mathcal C_1(G)$ is cofinal in $ \mathcal C(G)$ and $H_{alg}(\phi,U) \leq H_{alg}(\phi,U')$ whenever $U \subseteq U'$. 

\begin{Remark} We asked above the automorphism $\phi$ to be ``measure preserving". In this way one rules out many interesting 
cases of topological automorphisms that are not measure preserving (e.g., all automorphisms of $\R$ beyond $\pm \id_\R$). This condition is imposed in order to respect the definition of the morphisms in $\Se^*$. If one further relaxes
this condition on the morphisms in $\Se^*$ 
(without asking them to be contracting maps with respect to the pseudonorm),
 then one can obtain a semigroup entropy that covers the topological and the algebraic entropy of 
arbitrary topological automorphisms of locally compact groups (see \cite{DGV}  for more details). 
 \end{Remark}

\subsection{Algebraic $i$-entropy}

For a ring $R$ we denote by $\mod_R$ the category of right $R$-modules and $R$-module homomorphisms. We consider here the algebraic $i$-entropy introduced in \cite{SZ}, giving a functor ${\mathfrak{sub}_i}:\mod_R\to \SL$, to find $\ent_i$ from the general scheme. Here $i: \mod_R \to \R_+$ is an invariant of $\mod_R$ (i.e., $i(0)=0$ and $i(M) = i(N)$ whenever $M\cong N$).
Consider the following conditions: 
\begin{itemize}
\item[(a)] $i(N_1 + N_2)\leq i(N_1) + i(N_2)$ for all submodules $N_1$, $N_2$ of $M$;
\item[(b)] $i(M/N)\leq i(M)$ for every submodule $N$ of $M$;
\item[(b$^*$)] $i(N)\leq i(M)$ for every submodule $N$ of $M$. 
\end{itemize}
The invariant $i$ is called \emph{subadditive} if (a) and (b) hold, and it is called \emph{preadditive} if (a) and (b$^*$) hold.

\medskip
For $M\in\mod_R$ denote by $\La(M)$ the lattice of all submodules of $M$. The operations are intersection and sum of two submodules, the bottom element is $\{0\}$ and the top element is $M$. Now fix a subadditive invariant $i$ of $\mod_R$ and for a right $R$-module $M$ let
$$\sF_i(M)=\{\mbox{submodules $N$ of $M$ with }i(M)< \infty\},$$
which is a subsemilattice of $\La(M)$ ordered by inclusion. 
Define a norm on $\sF_i(M)$ setting $$v(H)=i(H)$$ for every $H \in \sF_i(M)$. The norm $v$ is not necessarily monotone (it is monotone if $i$ is both subadditive and preadditive).

For every homomorphism $\f: M \to N$ in $\mod_R$, 
\begin{center}
$\sF_i(\f): \sF_i(M) \to \sF_i(N)$, defined by $\sF_i(\f)(H) =\f(H)$, is a morphism in $\SL$.  
\end{center}
Moreover the norm $v$ makes the morphism $\sF_i(\f)$ contractive by the property (b) of the invariant. 
Therefore,
the assignments $M \mapsto \sF_i(M)$ and $\phi\mapsto \sF_i(\phi)$ define a covariant functor $$\mathfrak{sub}_i:\mod _R\to \SL.$$
We can conclude that, for a ring $R$ and a subadditive invariant $i$ of $\mod_R$,
$$h_{\mathfrak{sub}_i}=\ent_i.$$
 
If $i$ is preadditive, the functor ${\mathfrak{sub}_i}$ sends monomorphisms to embeddings and so $\ent_i$ is monotone under taking submodules.
If $i$ is both subadditive and preadditive then for every $R$-module $M$ the norm of ${\mathfrak{sub}_i}(M)$ is s-monotone, so $\ent_{i}$ satisfies also the Logarithmic Law.
In general this entropy is not monotone under taking quotients, but this can be obtained with stronger hypotheses on $i$ and with some restriction on the domain of ${\mathfrak{sub}_i}$.

\medskip
A clear example is given by vector spaces; the algebraic entropy $\ent_{\dim}$ for linear transformations of vector spaces was considered in full details in \cite{GBS}:

\begin{Example}
Let $K$ be a field. Then for every $K$-vector space $V$ let $\sF_d(M)$ be the set of all finite-dimensional subspaces $N$ of $M$. 

Then $(\sF_d(V),+)$ is a subsemilattice of $(\La(V),+)$ and $v(H)=\dim H$ defines a monotone norm on $\sF_d(V)$. For every morphism $\f: V \to W$ in $\mod_K$  
\begin{center}
the map $\sF_d(\f): \sF_d(V) \to \sF_d(W)$, defined by $H\mapsto\f(H)$, is a morphism in $\SL$.
\end{center}

Therefore, the assignments $M \mapsto \sF_d(M)$ and $\phi\mapsto \sF_d(\phi)$ define a covariant functor $$\mathfrak{sub}_d:\mod_K\to \SL.$$
Then $$h_{\mathfrak{sub}_d}=\ent_{\dim}.$$ 
Note that this entropy can be computed ad follows. Every flow $\phi: V \to V$ of $\mod_K$ can be considered as a $K[X]$-module $V_\phi$ letting $X$ act on $V$ as $\phi$. Then $h_{\mathfrak{sub}_d}(\phi)$ coincides with the rank of the $K[X]$-module $V_\phi$. 
\end{Example}

\subsection{Adjoint algebraic entropy}

We consider now again the category $\mathbf{Grp}$ of all groups and their homomorphisms, giving a functor $\mathfrak{sub}^\star:\mathbf{Grp}\to \SL$ such that the entropy defined using this functor coincides with the adjoint algebraic entropy $\ent^\star$ introduced in \cite{DGS}.

\medskip
For a group $G$ denote by $\sC(G)$ the family of all subgroups of finite index in $G$. It is a subsemilattice of $(\La(G), \cap)$. For $N\in\sC(G)$, let $$v(N) = \log[G:N];$$ then
\begin{center}
$(\sC(G),v)$ is a normed semilattice,
\end{center}
with neutral element $G$; moreover the norm $v$ is monotone.

For every group homomorphism $\f: G \to H$ 
\begin{center}
the map $\sC(\f): \sC(H) \to \sC(G)$, defined by $N\mapsto \f^{-1}(N)$, is a morphism in $\Se$. 
\end{center}
Then
the assignments $G\mapsto\sC(G)$ and $\phi\mapsto\sC(\phi)$ define a contravariant functor $$\mathfrak{sub}^\star:\mathbf{Grp}\to \SL.$$
Moreover
$$h_{\mathfrak{sub}^\star}=\ent^\star.$$

There exists also a version of the adjoint algebraic entropy for modules, namely the adjoint algebraic $i$-entropy $\ent_i^\star$ (see \cite{Vi}), which can be treated analogously.

\subsection{Topological entropy for totally disconnected compact groups}
  
Let $(G,\tau)$ be a totally disconnected compact group and consider the filter base $\V_G(1)$ of open subgroups of $G$.
Then
\begin{center}
$(\V_G(1), \cap)$ is a normed semilattice
\end{center}
with neutral element $G \in \V_G(1)$ and norm defined by $v_o(V)=\log [G:V]$ for every $V\in\V_G(1)$.

For a continuous homomorphism $\f: G\to H$ between compact groups, 
\begin{center}
the map $\V_H(1)\to \V_G(1)$, defined by $V \mapsto \f^{-1}(V)$, is a morphism in $\SL$.
\end{center}
This defines a contravariant functor $$\mathfrak{sub}_o^\star:\mathbf{TdCGrp}\to\SL,$$
which is a subfunctor of $\mathfrak{sub}^\star$.

Then the entropy $h_{\mathfrak{sub}^\star_o}$ coincides with the restriction to $\mathbf{TdCGrp}$ of the topological entropy $h_{top}$.
  
\medskip
This functor is related also to the functor $\mathfrak{cov}:\mathbf{TdCGrp} \to\Se$. Indeed, let $G$ be a totally disconnected compact group. Each $V\in \V_G(1)$ defines a cover $\U_V=\{x\cdot V\}_{x\in G}$ of $G$ with $v_o(V)=v(\U_V)$. So the map $V \mapsto \U_V$ defines an isomorphism between the normed semilattice $\mathfrak{sub}_o^\star(G)=\V_G(1)$ and the subsemigroup $\mathfrak{cov}_s(G)=\{\U_V:V \in \V_G(1)\}$ of $\mathfrak{cov}(G)$.

\subsection{Bridge Theorem}\label{BTsec}

In Definition \ref{BTdef} we have formalized the concept of Bridge Theorem between entropies $h_1:\mathfrak X_1 \to \R_+$ and $h_2:\mathfrak X_2 \to \R_+$ via functors $\varepsilon: \mathfrak X_1 \to \mathfrak X_2$.
Obviously, the Bridge Theorem with respect to the functor $\varepsilon$ is available when each $h_i$ has the form $h_i= h_{F_i}$ for appropriate functors $F_i:  \mathfrak{X}_i \to \Se$ ($i= 1,2$) that commute with $\varepsilon$ (i.e., $F_1 = F_2 \varepsilon$), that is 
$$
h_2(\varepsilon(\phi))= h_1(\phi)\  \mbox{  for all morphisms $\phi$ in }\  \mathfrak X_1.
$$  
Actually, it is sufficient that $F_i$ commute with $\varepsilon$ ``modulo $h_\Se$" (i.e., $h_\Se F_1 = h_\Se F_2 \varepsilon$) to obtain this conclusion: 
\begin{equation}\label{Buzz}
\xymatrix@R=6pt@C=37pt
{
\mathfrak{X}_1\ar[dd]_{\varepsilon}\ar[dr]^{F_1}\ar@/^2pc/[rrd]^{h_{1}}	&	&	\\
				& {\Se}\ar[r]|-{ {h_\Se}}&\R^+			\\
\mathfrak{X}_2\ar[ur]_{F_2}\ar@/_2pc/[rru]_{h_{2}}						&	&
}
\end{equation}

In particular the Pontryagin duality functor {$\ \widehat{}: \AG \to \CAG$} connects the category of abelian groups and that of compact abelian groups so connects the respective entropies $h_{alg}$ and $h_{top}$ by a Bridge Theorem. 
Taking the restriction to torsion abelian groups and  the totally disconnected compact groups one obtains: 

\begin{Theorem}[Weiss Bridge Theorem]\emph{\cite{W}}\label{WBT}
Let  $K$ be a totally disconnected compact abelian group and $\f: K\to K$ a continuous endomorphism. Then $h_{top}(\f) = \ent(\widehat \f)$. 
\end{Theorem} 
\begin{proof}
Since totally disconnected compact groups are zero-dimensional, 
every open finite cover $\mathcal U$ of $K$ admits a refinement consisting of clopen sets in $K$. 
Moreover, since $K$ admits a local base at 0 formed by open subgroups, it is possible to find a refinement of $\mathcal U$
of the form   $\mathcal U_V$ for some open subgroup $ \mathcal V$. This proves that $\mathfrak{cov}_s(K)$ is cofinal in $\mathfrak{cov}(K)$. Hence, we have 
$$
h_{top}(\phi)=h_\Se(\mathfrak{cov}(\phi))=h_\Se(\mathfrak{cov}_s(\phi)).
$$
Moreover, we have seen above that $\mathfrak{cov}_s(K)$ is isomorphic to $\mathfrak{sub}^\star_o(K)$, so
one can conclude that $$h_\Se(\mathfrak{cov}_s(\phi))=h_\Se(\mathfrak{sub}^\star_o (\phi)).$$
Now the semilattice isomorphism $L\to \mathcal F(\widehat K)$ given by $N \mapsto N^\perp$ preserves the norms, so it is an isomorphism in $\Se$. Hence $$h_\Se(\mathfrak{sub}^\star_o (\phi))=h_\Se(\mathfrak{sub}(\widehat \phi))$$ and consequently
$$h_{top}(\phi)= h_\Se(\mathfrak{sub}(\widehat \phi))=\ent(\widehat \phi).$$
\end{proof}

The proof of Weiss Bridge Theorem can be reassumed by the following diagram.
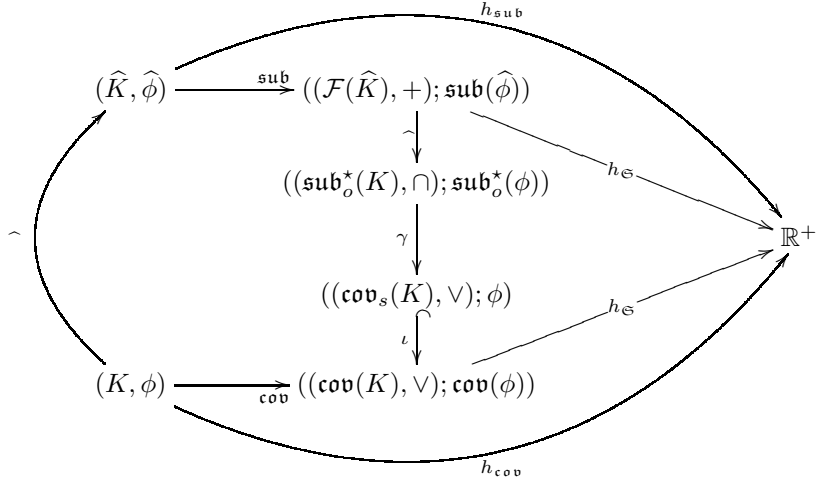
\begin{equation*}
\xymatrix@R=6pt@C=37pt
{
(\widehat K,\widehat\phi)\ar[r]^{\mathfrak{sub}}\ar@/^4.5pc/[rrrddd]^{h_{\mathfrak{sub}}}&((\sF(\widehat K),+);\mathfrak{sub}(\widehat \phi))\ar[dd]_{\widehat{}}\ar[dddrr]|-{h_\Se}& &\\
&	&	&\\
&((\mathfrak{sub}^\star_o(K),\cap);\mathfrak{sub}^\star_o(\phi))\ar[dd]_{\gamma}	&	&	\\
&	&	&	\R^+	\\
&((\cov_{s}(K),\vee);\phi)\ar@{^{(}->}[dd]_{\iota}		&	&	\\
&	 &	&\\
(K,\phi)\ar@/^3pc/[uuuuuu]^{\widehat{}\;\;}\ar[r]_{\mathfrak{cov}}\ar@/_4.5pc/[rrruuu]_{h_{\mathfrak{cov}}}&((\cov(K),\vee);\mathfrak{cov}(\phi))\ar[uuurr]|-{h_\Se} 	&	&
}						
\end{equation*}

Similar Bridge Theorems hold for other known entropies; they can be proved using analogous diagrams (see \cite{DGV1}).
The first one that we recall concerns the algebraic entropy $\ent$ and the adjoint algebraic entropy $\ent^\star$:

\begin{Theorem}
Let $\f: G\to G$ be an endomorphism of an abelian group. Then $\ent^\star(\f) = \ent(\widehat\f)$. 
\end{Theorem}

The other two Bridge Theorems that we recall here connect respectively the set-theoretic entropy $\mathfrak h$ with the topological entropy $h_{top}$ and the contravariant set-theoretic entropy $\mathfrak h^*$ with the algebraic entropy $h_{alg}$. 

We need to recall first the notion of generalized shift, which extend the Bernoulli shifts. For a map $\lambda:X\to Y$ between two non-empty sets 
and a fixed non-trivial group $K$, define $\sigma_\lambda:K^Y \to K^X$ by $\sigma_\lambda(f) = f\circ \lambda $ for $f\in K^Y$. For $Y = X$, 
$\lambda$ is a self-map of $X$ and $\sigma_\lambda$ was called \emph{generalized shift} of $K^X$
(see \cite{AADGH,AZD}). In this case $\bigoplus_X K$ is a $\sigma_\lambda$-invariant subgroup of $K^X$ precisely when $\lambda$ is finitely many-to-one. We denote $\sigma_\lambda\restriction_{\bigoplus_XK}$ by $\sigma_\lambda^\oplus$. 

\medskip
Item (a)  in the next theorem was proved in \cite{AZD} (see also \cite[Theorem 7.3.4]{DG-islam}) while item (b) is  \cite[Theorem 7.3.3]{DG-islam} (in the abelian case it was obtained in \cite{AADGH}).

\begin{Theorem} \emph{\cite{AZD}}
Let $K$ be a non-trivial finite group, let $X$ be a set and $\lambda:X\to X$ a self-map. 
\begin{itemize}
\item[(a)]Then $h_{top}(\sigma_\lambda)=\mathfrak h(\lambda)\log|K|$.
\item[(b)] If $\lambda$ is finite-to-one, then $h_{alg}(\sigma_\lambda^\oplus)=\mathfrak h^*(\lambda)\log|K|$.
\end{itemize}
\end{Theorem}

In terms of functors, fixed a non-trivial finite group $K$, let $\mathcal F_K: \mathbf{Set}\to \mathbf{TdCGrp}$ be the functor defined on flows,
sending a  non-empty set $X$ to $K^X$, $\emptyset$ to $0$, a self-map $\lambda:X\to X$ to $\sigma_\lambda:K^Y\to K^X$
when $X\ne \emptyset$. Then the pair $(\mathfrak h, h_{top})$ satisfies $(BT_{\mathcal F_K})$ with constant $\log |K|$.

Analogously, let $\mathcal G_K: \mathbf{Set}_\mathrm{fin}\to \mathbf{Grp}$ be the functor defined on flows sending $X$ to $\bigoplus_X K$ and a finite-to-one self-map $\lambda:X\to X$ to $\sigma_\lambda^\oplus:\bigoplus_X K\to \bigoplus_X K$. Then the pair $(\mathfrak h^*, h_{alg})$ satisfies $(BT_{\mathcal G_K})$ with constant $\log |K|$. 

\begin{Remark}
At the conference held in Porto Cesareo, R. Farnsteiner posed the following question related to the Bridge Theorem.
Is $h_{top}$ studied in non-Hausdorff compact spaces?

The question was motivated by the fact that the prime spectrum $\mathrm{Spec}(A)$ of a commutative ring $A$ is usually a non-Hausdorff compact space. Related to this question and to the entropy $h_\lambda$ defined for endomorphisms $\phi$ of local Noetherian rings $A$ (see \S \ref{NewSec1}), one may ask if there is any relation (e.g., a weak Bridge Theorem) between these two entropies and the functor $\mathrm{Spec}$; more precisely, one can ask whether there is any stable relation between $h_{top}(\mathrm{Spec}(\phi))$ and $h_\lambda(\phi)$.
\end{Remark}

\section{Algebraic entropy and its specific properties}\label{alg-sec}

In this section we give an overview of the basic properties of the algebraic entropy and the adjoint algebraic entropy. Indeed, we have seen that they satisfy the general scheme presented in the previous section, but on the other hand they were defined for specific group endomorphisms and these definitions permit to prove specific features, as we are going to briefly describe.
For further details and examples see \cite{DG}, \cite{DGS} and \cite{DG-islam}.

\subsection{Definition and basic properties}

Let $G$ be a group and $\phi:G\to G$ an endomorphism. For a finite subset $F$ of $G$, and for $n\in\N_+$, the \emph{$n$-th $\phi$-trajectory} of $F$ is 
\begin{equation*}\label{T_n}
T_n(\phi,F)=F\cdot\phi(F)\cdot\ldots\cdot\phi^{n-1}(F);
\end{equation*}
moreover let
\begin{equation}\label{gamma}
{\gamma_{\phi,F}(n)}=|T_n(\phi,F)|.
\end{equation}
The \emph{algebraic entropy of $\phi$ with respect to $F$} is 
\begin{equation*}\label{H}
H_{alg}(\phi,F)={\lim_{n\to \infty}\frac{\log \gamma_{\phi,F}(n) }{n}};
\end{equation*}
 This limit exists as $H_{alg}(\phi,F)=h_\Se(\mathcal H(\phi),F)$ and so Theorem \ref{limit} applies
(see also \cite{DG-islam} for a direct proof of the existence of this limit and \cite{DG} for the abelian case).
The \emph{algebraic entropy} of $\phi:G\to G$ is 
$$
h_{alg}(\phi)=\sup\{H_{alg}(\phi,F): F\ \text{finite subset of}\ G\}=h_\Se(\mathcal H(\phi)).
$$
Moreover
$$
\ent(\phi)=\sup\{H_{alg}(\phi,F): F\ \text{finite subgroup of}\ G\}.
$$
If $G$ is abelian, then $\ent(\phi)=\ent(\phi\restriction_{t(G)})= h_{alg}(\phi\restriction_{t(G)})$.

Moreover, $h_{alg}(\phi) = \ent(\phi)$ if $G$ is locally finite, that is every finite subset of $G$ generates a finite subgroup; note that every locally finite group is obviously torsion, while the converse holds true under the hypothesis that the group is abelian (but the solution of Burnside Problem shows that even groups of finite exponent fail to be locally finite).
 
\medskip
For every abelian group $G$, the identity map has $h_{alg}(\id_G)=0$ (as the normed semigroup $\mathcal H(G)$ is arithmetic, as seen above). Another basic example is given by the endomorphisms of $\Z$, indeed if $\f: \Z \to \Z$ is given by $\f(x) = mx$ for some positive integer $m$, then $h_{alg}(\f) = \log m$. 
The fundamental example for the algebraic entropy is the right Bernoulli shift:

\begin{Example}\label{shift}(Bernoulli normalization)
Let $K$ be a group.
\begin{itemize}
\item[(a)] The \emph{right Bernoulli shift} $\beta_K:K^{(\N)}\to K^{(\N)}$ is defined by
$$(x_0,\ldots,x_n,\ldots)\mapsto (1,x_0,\ldots,x_n,\ldots).$$
Then $h_{alg}(\beta_K)=\log|K|$, with the usual convention that $\log|K|=\infty$ when $K$ is infinite.
\item[(b)] The \emph{left Bernoulli shift} ${}_K\beta:K^{(\N)}\to K^{(\N)}$ is defined by
$$(x_0,\ldots,x_n,\ldots)\mapsto (x_1,\ldots,x_{n+1},\ldots).$$
Then $h_{alg}({}_K\beta)=0$, as ${}_K\beta$ is locally nilpotent.
\end{itemize}
\end{Example}

The following basic properties of the algebraic entropy are consequences of the general scheme and were proved directly in \cite{DG-islam}.

\begin{fact}\label{properties}
\emph{Let $G$ be a group and $\phi:G\to G$ an endomorphism.
\begin{itemize}
\item[(a)]\emph{[Invariance under conjugation]} If $\phi=\xi^{-1}\psi\xi$, where $\psi:H\to H$ is an endomorphism and $\xi:G\to H$ isomorphism, then $h_{alg}(\phi) = h_{alg}(\psi)$.
\item[(b)]\emph{[Monotonicity]} If $H$ is a $\phi$-invariant normal subgroup of the group $G$, and $\overline\phi:G/H\to G/H$ is the endomorphism induced by $\phi$, then $h_{alg}(\phi)\geq \max\{h_{alg}(\phi\restriction_H),h_{alg}(\overline{\phi})\}$.
\item[(c)]\emph{[Logarithmic Law]} For every $k\in\N$ we have $h_{alg}(\phi^k) = k \cdot h_{alg}(\phi)$; if $\phi$ is an automorphism, then $h_{alg}(\phi)=h_{alg}(\phi^{-1})$, so $h_{alg}(\phi^k)=|k|\cdot h_{alg}(\phi)$ for every $k\in\Z$.
\item[(d)]\emph{[Continuity]} If $G$ is direct limit of $\phi$-invariant subgroups $\{G_i : i \in I\}$, then $h_{alg}(\phi)=\sup_{i\in I}h_{alg}(\phi\restriction_{G_i}).$
\item[(e)]\emph{[Weak Addition Theorem]} If $G=G_1\times G_2$ and $\phi_i:G_i\to G_i$ is an endomorphism for $i=1,2$, then $h_{alg}(\phi_1\times\phi_2)=h_{alg}(\phi_1)+h_{alg}(\phi_2)$.
\end{itemize}
}
\end{fact}

As described for the semigroup entropy in the previous section, and as noted in \cite[Remark 5.1.2]{DG-islam}, for group endomorphisms $\phi:G\to G$ it is possible to define also a ``left'' algebraic entropy, letting for a finite subset $F$ of $G$, and for $n\in\N_+$,
$$T_n^\#(\phi,F)=\phi^{n-1}(F)\cdot\ldots\cdot\phi(F)\cdot F,$$
$$H^\#_{alg}(\phi,F)={\lim_{n\to \infty}\frac{\log |T^\#_n(\phi,F)|}{n}}$$
and
$$h_{alg}^\#(\phi)=\sup\{H^\#_{alg}(\phi,F): F\ \text{finite subset of}\ G\}.$$
Answering a question posed in \cite[Remark 5.1.2]{DG-islam}, we see now that $$h_{alg}(\phi)=h_{alg}^\#(\phi).$$
Indeed, every finite subset of $G$ is contained in a finite subset $F$ of $G$ such that $1\in F$ and $F={F^{-1}}$; for such $F$ we have
$$H_{alg}(\phi,F)=H_{alg}^\#(\phi,F),$$
since, for every $n\in\N_+$,
\begin{equation*}\begin{split}
T_n(\phi,F)^{-1}=\phi^{n-1}(F)^{-1}\cdot\ldots\cdot\phi(F)^{-1}\cdot F^{-1}=\\ \phi^{n-1}(F^{-1})\cdot\ldots\cdot\phi(F^{-1})\cdot F^{-1}=T_n^\#(\phi,F)
\end{split}\end{equation*}
and so $|T_n(\phi,F)|=|T_n(\phi,F)^{-1}|=|T_n^\#(\phi,F)|$.

\subsection{Algebraic Yuzvinski Formula, Addition Theorem and Uni\-que\-ness}\label{ab-sec}

We recall now some of the main deep properties of the algebraic entropy in the abelian case. They are not consequences of the general scheme and are proved using the specific features of the algebraic entropy coming from the definition given above. We give here the references to the papers where these results were proved, for a general exposition on algebraic entropy see the survey paper \cite{DG-islam}.

\medskip
The next proposition shows that the study of the algebraic entropy for torsion-free abelian groups can be reduced to the case of divisible ones. It was announced for the first time by Yuzvinski \cite{Y1}, for a proof see \cite{DG}. 

\begin{Proposition}\label{AA_} 
Let $G$ be a torsion-free abelian group, $\phi:G\to G$ an endomorphism and denote by $\widetilde\phi$ the (unique) extension of $\phi$ to the divisible hull $D(G)$ of $G$. Then $h_{alg}(\phi)=h_{alg}(\widetilde\phi)$.
\end{Proposition}

Let $f(t)=a_nt^n+a_1t^{n-1}+\ldots+a_0\in\Z[t]$ be a primitive polynomial and let $\{\lambda_i:i=1,\ldots,n\}\subseteq\mathbb C$ be the set of all roots of $f(t)$.
The \emph{(logarithmic) Mahler measure} of $f(t)$ is $$m(f(t))= \log|a_n| + \sum_{|\lambda_i|>1}\log |\lambda_i|.$$
The Mahler measure plays an important role in number theory and arithmetic geometry and is involved in the famous Lehmer Problem, asking whether $\inf\{m(f(t)):f(t)\in\Z[t]\ \text{primitive}, m(f(t))>0\}>0$ (for example see \cite{Ward0} and \cite{Hi}).

If $g(t)\in\Q[t]$ is monic, then there exists a smallest positive integer $s$ such that $sg(t)\in\Z[t]$; in particular, $sg(t)$ is primitive. The Mahler measure of $g(t)$ is defined as $m(g(t))=m(sg(t))$. Moreover, if $\phi:\Q^n\to \Q^n$ is an endomorphism, its characteristic polynomial $p_\phi(t)\in\Q[t]$ is monic, and the Mahler measure of $\phi$ is $m(\phi)=m(p_\phi(t))$.

\medskip
The  formula \eqref{yuzeq} below was given a direct proof recently in \cite{GV}; it is the algebraic counterpart of the so-called Yuzvinski Formula for the topological entropy \cite{Y1} (see also \cite{LW}). It gives the values of the algebraic entropy of linear transformations of finite dimensional rational vector spaces in terms of the Mahler measure, so it allows for a connection of the algebraic entropy with Lehmer Problem.

\begin{Theorem}[Algebraic Yuzvinski Formula] \label{AYF}\emph{\cite{GV}}
Let $n\in\N_+$ and $\phi:\Q^n\to\Q^n$ an endomorphism. Then 
\begin{equation}\label{yuzeq}
h_{alg}(\phi)=m(\phi).
\end{equation}
\end{Theorem}

The next property of additivity of the algebraic entropy was first proved for torsion abelian groups in \cite{DGSZ}, while the proof of the general case was given in \cite{DG} applying the Algebraic Yuzvinski Formula.

\begin{Theorem}[Addition Theorem]\emph{\cite{DG}}\label{AT}
Let $G$ be an abelian group, $\phi:G\to G$ an endomorphism, $H$ a $\phi$-invariant subgroup of $G$ and $\overline\phi:G/H\to G/H$ the endomorphism induced by $\phi$. Then $$h_{alg}(\phi)=h_{alg}(\phi\restriction_H)+ h_{alg}(\overline\phi).$$
\end{Theorem}

Moreover, uniqueness is available for the algebraic entropy in the category of all abelian groups. As in the case of the Addition Theorem, also the Uniqueness Theorem was proved in general in \cite{DG}, while it was previously proved in \cite{DGSZ} for torsion abelian groups.

\begin{Theorem}[Uniqueness Theorem]\label{UT}\emph{\cite{DG}}
The algebraic entropy $$h_{alg}:\mathrm{Flow}_\mathbf{AbGrp}\to\R_+$$ is the unique function such that:
\begin{itemize}
\item[(a)] $h_{alg}$ is invariant under conjugation;
\item[(b)] $h_{alg}$ is continuous on direct limits;
\item[(c)] $h_{alg}$ satisfies the Addition Theorem;
\item[(d)] for $K$ a finite abelian group, $h_{alg}(\beta_K)=\log|K|$;
\item[(e)] $h_{alg}$ satisfies the Algebraic Yuzvinski Formula.
\end{itemize}
\end{Theorem}


\subsection{The growth of a finitely generated flow in $\mathbf{Grp}$}\label{Growth-sec}

In order to measure and classify the growth rate of maps $\N \to \N$, one need the relation $\preceq$ defined as follows. 
For $\gamma, \gamma': \N \to \N$ let $\gamma \preceq \gamma'$ if there exist $n_0,C\in\N_+$ such that $\gamma(n) \leq \gamma'(Cn)$ for every $n\geq n_0$. 
Moreover $\gamma\sim\gamma$  if $\gamma\preceq\gamma'$ and $\gamma'\preceq\gamma$ (then $\sim$ is an equivalence relation), and $\gamma\prec\gamma'$ if $\gamma\preceq\gamma'$ but $\gamma\not\sim\gamma'$.

For example, for every $\alpha, \beta\in\R_{\geq0}$, $n^\alpha\sim n^\beta$ if and only if $\alpha=\beta$; if $p(t)\in\Z[t]$ and $p(t)$ has degree $d\in\N$, then $p(n)\sim n^d$.
On the other hand, $a^n\sim b^n$ for every $a,b\in\R$ with $a,b>1$, so in particular all exponentials are equivalent with respect to $\sim$. 

So a map $\gamma: \N \to \N$ is called:
\begin{itemize}
\item[(a)] \emph{polynomial} if $\gamma(n) \preceq n^d$ for some $d\in\N_+$;
\item[(b)] \emph{exponential} if $\gamma(n) \sim 2^n$;
\item[(c)] \emph{intermediate} if $\gamma(n)\succ n^d$ for every $d\in\N_+$ and $\gamma(n)\prec 2^n$.
\end{itemize}

Let $G$ be a group, $\phi:G\to G$ an endomorphism and $F$ a non-empty finite subset of $G$.
Consider the function, already mentioned in \eqref{gamma}, 
$$
\gamma_{\phi,F}:\N_+\to\N_+\ \text{defined by}\ \gamma_{\phi,F}(n)=|T_n(\phi,F)|\ \text{for every}\ n\in\N_+.
$$
Since 
$$
|F|\leq\gamma_{\phi,F}(n)\leq|F|^n\mbox{ for every }n\in\N_+,
$$
the growth of $\gamma_{\phi,F}$ is always at most exponential; moreover, $H_{alg}(\phi,F)\leq \log |F|$.
So, following \cite{DG0} and \cite{DG-islam}, we say that $\phi$ has \emph{polynomial} (respectively, \emph{exponential}, \emph{intermediate}) \emph{growth at $F$} if $\gamma_{\phi,F}$ is polynomial (respectively, exponential, intermediate).

\medskip
Before proceeding further, let us make an important point here. All properties considered above concern practically the $\phi$-invariant 
subgroup $G_{\phi,F}$ of $G$ generated by the trajectory $T(\phi, F) = \bigcup_{n\in\N_+} T_n(\phi, F)$ and the restriction $\phi\restriction_{G_{\phi,F}}$.

\begin{Definition}
We say that the flow $(G, \phi)$ in $\mathbf{Grp}$ is \emph{finitely generated} if $G = G_{\phi,F}$ for some finite subset $F$ of $G$.
\end{Definition}

Hence, all properties listed above concern finitely generated flows in $\mathbf{Grp}$. 
We conjecture the following, knowing that it holds true when $G$ is abelian or when $\phi=\id_G$: if the flow $(G,\phi)$ is finitely generated, and if $G = G_{\phi,F}$  and $G = G_{\phi,F'}$ for some finite subsets $F$ and $F'$ of $G$, then 
$\gamma_{\phi,F}$ and $\gamma_{\phi,F'}$ have the same type of growth.
In this case the growth of a finitely generated flow $G_{\phi,F}$ would not depend on the specific finite set of generators $F$ (so $F$ can always be taken symmetric). In particular, one could speak of growth of a finitely generated flow without any reference to a specific finite set of generators. Nevertheless, one can give in general the following

\begin{Definition}
Let $(G,\phi)$ be a finitely generated flow in $\mathbf{Grp}$.
We say that $(G,\phi)$ has
\begin{itemize}
\item[(a)] \emph{polynomial growth} if $\gamma_{\phi,F}$ is polynomial
for every finite subset $F$ of $G$;
\item[(b)] \emph{exponential growth} if there exists a finite subset
$F$ of $G$ such that $\gamma_{\phi,F}$ is exponential;
\item[(c)] \emph{intermediate growth} otherwise.
\end{itemize}
We denote by $\mathrm{Pol}$ and $\mathrm{Exp}$ the classes of finitely generated flows in $\mathbf{Grp}$ of polynomial and exponential growth respectively. Moreover, $\mathcal M=\mathrm{Pol}\cup\mathrm{Exp}$ is the class of finitely generated flows of non-intermediate growth. 
\end{Definition}

This notion of growth generalizes the classical one of growth of a finitely generated group given independently by Schwarzc \cite{Sch} and Milnor \cite{M1}.
Indeed, if $G$ is a finitely generated group and $X$ is a finite symmetric set of generators of $G$, then $\gamma_X=\gamma_{\id_G,X}$ is the classical \emph{growth function} of $G$ with respect to $X$. For a connection of the terminology coming from the theory of algebraic entropy and the classical one, note that for $n\in\N_+$ we have $T_n(\id_G,X)=\{g\in G:\ell_X(g)\leq n\}$, where $\ell_X(g)$ 
 is the length of the shortest word $w$ in the alphabet $X$ such that $w=g$ (see \S \ref{NewSec1} (c)). Since $\ell_X$ is a norm on $G$, $T_n(\id_G,X)$ is the ball of radius $n$ centered at $1$ and $\gamma_X(n)$ is the cardinality of this ball.

\medskip
Milnor \cite{M3} proposed the following problem on the growth of finitely generated groups.

\begin{problem}[Milnor Problem]\label{Milnor-pb}{\cite{M3}}
Let $G$ be a finitely generated group and $X$ a finite set of generators of $G$.
\begin{itemize}
\item[(i)] Is the growth function $\gamma_X$ necessarily equivalent either to a power of $n$ or to the exponential function $2^n$?
\item[(ii)] In particular, is the {growth exponent} $\delta_G=\limsup_{n\to \infty}\frac{\log\gamma_X(n)}{\log n}$ either a well defined integer or infinity?
For which groups is $\delta_G$ finite?
\end{itemize}
\end{problem}

Part (i) of Problem \ref{Milnor-pb} was solved negatively by Grigorchuk in \cite{Gri1,Gri2,Gri3,Gri4}, where he constructed his famous examples of finitely generated groups $\mathbb G$ with intermediate growth. 
For part (ii) Milnor conjectured that $\delta_G$ is finite if and only if $G$ is virtually nilpotent (i.e., $G$ contains a nilpotent finite-index subgroup). The same conjecture was formulated by Wolf \cite{Wolf} (who proved that a nilpotent finitely generated group has polynomial growth) and Bass \cite{Bass}.  Gromov \cite{Gro} confirmed Milnor's conjecture: 

\begin{Theorem}[Gromov Theorem]\label{GT}\emph{\cite{Gro}}
A finitely generated group $G$ has polynomial growth if and only if $G$ is virtually nilpotent.
\end{Theorem}

The following two problems on the growth of finitely generated flows of groups are inspired by Milnor Problem.

\begin{problem}
Describe the permanence properties of the class $\mathcal M$.
\end{problem}

Some stability properties of the class $\mathcal M$ are easy to check. For example, stability under taking finite direct products is obviously available, while stability under taking subflows (i.e., invariant subgroups) and factors fails even in the classical case of identical flows. Indeed, Grigorchuk's group $\mathbb G$ is a quotient of a finitely generated free group $F$, that has exponential growth; so $(F,\id_F) \in \mathcal M$, while $(\mathbb G, \id_{\mathbb G})\not \in\mathcal M$. Furthermore, letting $G = \mathbb G \times F$, one has $(G,\id_G) \in \mathcal M$, while $(\mathbb G, \id_{\mathbb G})\not \in\mathcal M$, so $\mathcal M$ is not stable even under taking direct summands.
On the other hand, stability under taking powers is available since $(G,\phi) \in \mathcal M$ if and only if $(G,\phi^n) \in \mathcal M$ for $n\in\N_+$.

\begin{problem}\label{Ques4}
\begin{itemize}
\item[(i)] Describe the finitely generated groups $G$ such that $(G,\phi)\in\mathcal M$  for every endomorphism $\phi:G\to G$.
\item[(ii)] Does there exist a finitely generated group $G$ such that $(G,\id_G)\in\mathcal M$ but $(G,\phi)\not\in\mathcal M$ for some endomorphism $\phi:G\to G$?
\end{itemize}
\end{problem}

In item (i) of the above problem we are asking to describe all finitely generated groups $G$ of non-intermediate growth such that $(G,\phi)$ has still non-intermediate growth for every endomorphism $\phi:G\to G$. On the other hand, in item (ii) we ask to find a finitely generated group $G$ of non-intermediate growth that admits an endomorphism $\phi:G\to G$ of intermediate growth.

\medskip
The basic relation between the growth and the algebraic entropy is given by Proposition \ref{exp} below. For a finitely generated group $G$, 
an endomorphism $\phi$ of $G$ and a pair $X$ and $X'$ of finite generators of $G$, one has $\gamma_{\phi,X}\sim\gamma_{\phi,X'}$.
Nevertheless, $H_{alg}(\phi,X)\neq H_{alg}(\phi,X')$ may occur; in this case $(G,\phi)$ has necessarily exponential growth. We give two examples to this effect:

\begin{Example}\label{exaAugust}
\begin{itemize}
\item[(a)] {\cite{DG-islam}} Let $G$ be the free group with two generators $a$ and $b$; then $X=\{a^{\pm 1},b^{\pm 1}\}$ gives $H_{alg}(\id_G,X)=\log 3$ while for $X'=\{a^{\pm 1},b^{\pm 1},(ab)^{\pm 1}\}$ we have $H_{alg}(\id_G,X')=\log 4$.
\item[(b)] Let $G = \Z$ and $\f: \Z \to \Z$ defined by $\f(x) = mx$ for every $x\in \Z$ and with $m>3$. Let also $X= \{0,\pm 1\}$ and $X'= \{0,\pm 1, \ldots \pm m\}$. Then $H_{alg}(\f,X) \leq \log |X| =\log 3$, while $H_{alg}(\f,X')= h_{alg}(\f)= \log m$. 
\end{itemize}
\end{Example}

\begin{Proposition}\label{exp}\emph{\cite{DG-islam}}
Let $(G,\phi)$ be a finitely generated flow in {\bf Grp}.
\begin{itemize}
\item[(a)]Then $h_{alg}(\phi)>0$ if and only if $(G,\phi)$ has exponential growth.
\item[(b)]If $(G,\phi)$ has polynomial growth, then $h_{alg}(\phi)=0$.
\end{itemize}
\end{Proposition}

In general the converse implication in item (b) is not true even for the identity. Indeed, 
if $(G,\phi)$ has intermediate growth, then $h_{alg}(\phi)=0$ by item (a). So for Grigorchuk's group $\mathbb G$, the flow $(\mathbb G,\id_\mathbb G)$ has intermediate growth yet $h_{alg}(\id_\mathbb G)=0$.
This motivates the following

\begin{Definition}\label{MPara}
Let $\mathcal G$ be a class of groups and $\Phi$ be a class of morphisms. We say that the pair $(\mathcal G, \Phi)$ satisfies Milnor Paradigm (briefly, MP) if no finitely generated flow $(G,\phi)$ with $G\in\mathcal G$ and $\phi\in\Phi$ can have intermediate growth. 
\end{Definition}

In terms of the class $\mathcal M$, $$(\mathcal G, \Phi)\  \text{satisfies MP if and only if }\ (\mathcal G, \Phi)\in \mathcal M\ (\forall G\in\mathcal G)(\forall \phi\in\Phi) .$$
Equivalently, $(\mathcal G, \Phi)$ satisfies MP when $h_{alg}(\phi)=0$ always implies that $(G,\phi)$ has polynomial growth for finitely generated flows $(G,\phi)$ with $G\in\mathcal G$ and $\phi\in\Phi$. 

In these terms Milnor Problem \ref{Milnor-pb} (i) is asking whether the pair $(\mathbf{Grp},\mathcal{I}d)$ satisfies MP, where $\mathcal I d$ is the class of all identical endomorphisms. So we give the following general open problem.
 
\begin{problem}\label{PB0}
\begin{itemize}
  \item[(i)] Find pairs $(\mathcal G,\Phi)$ satisfying MP.
  \item[(ii)] For a given $\Phi$ determine the properties of the largest class  $\mathcal G_\Phi$ such that $(\mathcal G_\Phi, \Phi)$ satisfies MP.
  \item[(iii)] For a given $\mathcal G$ determine the properties of the largest class $\Phi_{\mathcal G}$ such that $(\mathcal G, \Phi_{\mathcal G})$ satisfies MP. 
  \item[(iv)] Study the Galois correspondence between classes of groups  $\mathcal G$ and classes of endomorphisms $\Phi$ determined by MP. 
\end{itemize}
\end{problem}

According to the definitions, the class $\mathcal G_{\mathcal{I} d}$ coincides with the class of finitely generated groups of non-intermediate growth. 

\medskip
The following result solves Problem \ref{PB0} (iii) for $\mathcal G=\mathbf{AbGrp}$, showing that $\Phi_\mathbf{AbGrp}$ coincides with the class $\mathcal E$ of all endomorphisms. 

\begin{Theorem}[Dichotomy Theorem]\emph{\cite{DG0}}\label{DT} 
There exist no finitely generated flows of intermediate growth in $\mathbf{AbGrp}$.
\end{Theorem}

Actually, one can extend the validity of this theorem to nilpotent groups. This leaves open the following particular case of Problem \ref{PB0}. We shall see in Theorem \ref{osin} that the answer to (i) is positive when $\phi=\id_G$. 

\begin{question}\label{Ques1}
Let $(G,\phi)$ be a finitely generated flow in $\mathbf{Grp}$.
\begin{itemize}
\item[(i)] If $G$ is solvable, does $(G,\phi)\in\mathcal M$?
\item[(ii)] If $G$ is a free group, does $(G,\phi)\in\mathcal M$?
\end{itemize}
\end{question}



We state now explicitly a particular case of Problem \ref{PB0}, inspired by the fact that the right Bernoulli shifts have no non-trivial quasi-periodic points and they have uniform exponential growth (see Example \ref{bern}). In \cite{DG0} group endomorphisms $\phi:G\to G$ without non-trivial quasi-periodic points are called algebraically ergodic for their connection (in the abelian case and through Pontryagin duality) with ergodic transformations of compact groups.

\begin{question}\label{Ques2}
Let $\Phi_0$ be the class of endomorphisms without non-trivial quasi-periodic points.  Is it true that the pair $(\mathbf{Grp},\Phi_0)$ satisfies MP? 
\end{question}




\medskip
For a finitely generated group $G$, the \emph{uniform exponential growth rate} of $G$ is defined as
$$
\lambda(G)=\inf\{H_{alg}(\id_G,X):X\ \text{finite set of generators of}\ G\}
$$
 (see for instance \cite{dlH-ue}). Moreover, $G$ has \emph{uniform exponential growth} if $\lambda(G)>0$.
Gromov \cite{GroLP} asked whether every finitely generated group of exponential growth is also of uniform exponential growth. This problem was recently solved by Wilson \cite{Wilson} in the negative.

Since the algebraic entropy of a finitely generated flow $(G,\phi)$ in $\mathbf{Grp}$  can be computed as
$$
h_{alg}(\phi)=\sup\{H_{alg}(\phi,F): F\ \text{finite subset of $G$ such that $G=G_{\phi,F}$}\},
$$
one can give the following counterpart of the uniform exponential growth rate for flows: 

\begin{Definition}
For $(G,\phi)$ be a finitely generated flow in $\mathbf{Grp}$ let
$$
\lambda(G,\phi)=\inf\{H_{alg}(\phi,F): F\ \text{finite subset of $G$ such that $G=G_{\phi,F}$} \}.
$$
The flow $(G,\phi)$ is said to have \emph{uniformly exponential growth} if $\lambda(G,\phi)>0$. 

Let $\mathrm{Exp}_\mathrm u$ be the subclass of $\mathrm{Exp}$ of all finitely generated flows in $\mathbf{Grp}$ of uniform exponential growth.
\end{Definition}

Clearly $\lambda(G,\phi)\leq h_{alg}(\phi)$, so one has the obvious implication
\begin{equation}\label{GP}
h_{alg}(\phi)=0\ \Rightarrow\ \lambda(G,\phi)=0.
\end{equation}

To formulate the counterpart of Gromov's problem on uniformly exponential growth it is worth to isolate also the class 
$\mathcal W$ of the finitely generated flows in $\mathbf{Grp}$ of exponential but not uniformly exponential growth
 (i.e., $\mathcal W=\mathrm{Exp}\setminus \mathrm{Exp}_\mathrm u$). Then $\mathcal W$ is the class of finitely generated flows $(G,\phi)$ in $\mathbf{Grp}$ for which \eqref{GP} cannot be inverted, namely $h_{alg}(\phi)> 0=\lambda(G,\phi)$.  

\medskip
We start stating the following problem.

\begin{problem}
Describe the permanence properties of the classes $\mathrm{Exp}_\mathrm u$ and $\mathcal W$.
\end{problem}

It is easy to check that $\mathrm{Exp}_\mathrm u$ and $\mathcal W$ are stable under taking direct products. On the other hand, stability of $\mathrm{Exp}_\mathrm u$ under taking subflows (i.e., invariant subgroups) and factors fails even in the classical case of identical flows. Indeed, Wilson's group $\mathbb W$ is a quotient of a finitely generated free group $F$, that has uniform exponential growth (see \cite{dlH-ue}); so $(F,\id_F)\in \mathrm{Exp}_\mathrm u$, while $(\mathbb W, \id_{\mathbb W})\in\mathcal W$. Furthermore, letting $G = \mathbb W \times F$, one has $(G,\id_G)\in \mathrm{Exp}_\mathrm u$, while $(\mathbb W, \id_{\mathbb W})\in\mathcal W$, so $\mathrm{Exp}_\mathrm u$ is not stable even under taking direct summands.

\medskip
%

In the line of MP, introduced in Definition \ref{MPara}, we can formulate also the following

\begin{Definition}\label{GPara}
Let $\mathcal G$ be a class of groups and $\Phi$ be a class of morphisms. We say that the pair $(\mathcal G, \Phi)$ satisfies Gromovr Paradigm (briefly, MP), if every finitely generated flow $(G,\phi)$ with $G\in\mathcal G$ and $\phi\in\Phi$ of exponential growth has has uniform exponential growth. 
\end{Definition}

In terms of the class $\mathcal W$, 
$$
(\mathcal G, \Phi)\  \text{satisfies GP if and only if }\ (\mathcal G, \Phi)\not\in \mathcal M\ (\forall G\in\mathcal G)(\forall \phi\in\Phi) .
$$
In these terms, Gromov's problem on uniformly exponential growth asks whether the pair $(\mathbf{Grp}, \mathcal I d)$ satisfies GP. In analogy to the general Problem \ref{PB0}, one can consider the following obvious counterpart for GP: 

\begin{problem}\label{PB1}
\begin{itemize}
\item[(i)] Find pairs $(\mathcal G, \Phi)$ satisfying GP.
\item[(ii)] For a given $\Phi$ determine the properties of the largest class  $\mathcal G_\Phi$ such that $(\mathcal G_\Phi,\Phi)$ satisfies GP. 
\item[(iii)] For a given $\mathcal G$ determine the properties of the largest class $\Phi_{\mathcal G}$ such that $(\mathcal G,\Phi_\mathcal G)$ satisfies GP. 
\item[(iv)] Study the Galois correspondence between classes of groups  $\mathcal G$ and classes of endomorphisms $\Phi$
determined by GP. 
\end{itemize}
\end{problem}

We see now in item (a) of the next example a particular class of finitely generated flows for which $\lambda$ coincides with $h_{alg}$ and they are both positive, so in particular these flows are all in $\mathrm{Exp}_\mathrm u$. In item (b) we leave an open question related to Question \ref{Ques2}.

\begin{Example}\label{bern}
\begin{itemize}
\item[(a)] For a finite group $K$, consider the flow $(\bigoplus_\N K,\beta_K)$. We have seen in Example \ref{shift} that $h_{alg}(\beta_K)=\log|K|$. In this case we have $\lambda(\bigoplus_\N K,\beta_K)=\log|K|$, since a subset $F$ of $\bigoplus_\N K$ generating the flow $(\bigoplus_\N K,\beta_K)$ must contain the first copy $K_0$ of $K$ in $\bigoplus_\N K$, and $H_{alg}(\beta_K,K_0)=\log|K|$.
\item[(b)] Is it true that $\lambda(G,\f) = h_{alg}(\f) > 0$ for every finitely generated flow $(G,\phi)$ in $\mathbf{Grp}$ such that $\f \in \Phi_0$?
In other terms, we are asking whether all finitely generated flows $(G,\phi)$ in $\mathbf{Grp}$ with $\phi\in\Phi_0$ have uniform exponential growth (i.e., are contained in $\mathrm{Exp}_\mathrm u$).
\end{itemize}
\end{Example}

%

\medskip
One can also consider the pairs $(\mathcal G, \Phi)$ satisfying the conjunction MP \& GP. 
For any finitely generated flow $(G,\phi)$ in $\mathbf{Grp}$ one has 
\begin{equation}\label{osin-eq}
(G,\phi)\ \text{has polynomial growth}\ \ \buildrel{(1)}\over\Longrightarrow\ h_{alg}(\phi)=0\  \ \buildrel{(2)}\over\Longrightarrow\ \lambda(G,\phi)=0.
\end{equation}
The converse implication of (1) (respectively, (2)) holds for all $(G,\phi)$  with $G\in\mathcal G$ and $\phi\in\Phi$ 
precisely when the pair $(\mathcal G, \Phi)$ satisfies MP (respectively, GP).
 Therefore,   the pair $(\mathcal G, \Phi)$ satisfies the conjunction MP \& GP precisely when 
 the three conditions in \eqref{osin-eq} are all equivalent (i.e., $\lambda(G,\phi)=0 \Rightarrow (G,\phi)\in \mathrm{Pol}$) for all finitely generated flows $(G,\phi)$  with $G\in\mathcal G$ and $\phi\in\Phi$. 

A large class of groups $\mathcal G$ such that $(\mathcal G, \mathcal I d)$ satisfies MP \& GP was found
by Osin \cite{O} who proved that a finitely generated solvable group $G$ of zero uniform exponential growth is virtually nilpotent, and recently this result was generalized in \cite{O1} to elementary amenable groups. Together with Gromov Theorem and Proposition \ref{exp}, this gives immediately the following

\begin{Theorem}\label{osin}
Let $G$ be a finitely generated 
 elementary amenable group. The following conditions are equivalent:
\begin{itemize}
\item[(a)] $h_{alg}(\id_G)=0$;
\item[(b)] $\lambda(G)=0$;
\item[(c)] $G$ is virtually nilpotent;
\item[(d)] $G$ has polynomial growth.
\end{itemize}
\end{Theorem}

This theorem shows that the pair $\mathcal G=\{\mbox{elementary amenable groups}\}$ and $\Phi =\mathcal{I} d$ satisfies simultaneously MP and GP. In other words it proves that the three conditions in \eqref{osin-eq} are all equivalent when $G$ is an elementary amenable finitely generated group and $\phi=\id_G$. 

\subsection{Adjoint algebraic entropy}\label{aent-sec}

We recall here the definition of the adjoint algebraic entropy $\ent^\star$ and we state some of its specific features not deducible from the general scheme, so beyond the ``package" of general properties coming from the equality $\ent^\star=h_{\mathfrak{sub}^\star}$ such as Invariance under conjugation and inversion, Logarithmic Law, Monotonicity for factors (these properties were proved in \cite{DG-islam} in the general case and previously in \cite{DGS} in the abelian case applying the definition).

\medskip
In analogy to the algebraic entropy $\ent$, in \cite{DGS} the adjoint algebraic entropy of endomorphisms of abelian groups $G$ was introduced ``replacing" the family $\mathcal F(G)$ of all finite subgroups of $G$ with the family $\mathcal C(G)$ of all finite-index subgroups of $G$. The same definition was extended in \cite{DG-islam} to the more general setting of endomorphisms of arbitrary groups as follows.
Let $G$ be a group and $N\in \mathcal C(G)$. For an endomorphism $\phi:G\to G$ and $n\in\N_+$, the \emph{$n$-th $\phi$-cotrajectory of $N$} is $$C_n(\phi,N)=N\cap\phi^{-1}(N)\cap\ldots\cap\phi^{-n+1}(N).$$ 
The \emph{adjoint algebraic entropy of $\phi$ with respect to $N$} is 
$$
H^\star(\phi,N)={\lim_{n\to \infty}\frac{\log[G:C_n(\phi,N)]}{n}}.
$$
This limit exists as $H^\star(\phi,N)=h_\Se(\mathcal C(\phi),N)$ and so Theorem \ref{limit} applies.
The \emph{adjoint algebraic entropy of $\phi$} is $$\ent^\star(\phi)=\sup\{H^\star(\phi,N):N\in\mathcal C(G)\}.$$


The values of the adjoint algebraic entropy of the Bernoulli shifts were calculated in \cite[Proposition 6.1]{DGS} applying \cite[Corollary 6.5]{G0} and the Pontryagin duality; a direct computation can be found in \cite{G}. So, in contrast with what occurs for the algebraic entropy, we have:

\begin{Example}[Bernoulli shifts]\label{beta*}
For $K$ a non-trivial group, $$\ent^\star(\beta_K)=\ent^\star({}_K\beta)=\infty.$$
\end{Example}

%

As proved in \cite{DGS}, the adjoint algebraic entropy satisfies the Weak Addition Theorem, while the Monotonicity for invariant subgroups fails even for torsion abelian groups; in particular, the Addition Theorem fails in general. On the other hand, the Addition Theorem holds for bounded abelian groups:

\begin{Theorem}[Addition Theorem]\label{AT*} 
Let $G$ be a bounded abelian group, $\phi:G\to G$ an endomorphism, $H$ a $\phi$-invariant subgroup of $G$ and $\overline\phi:G/H\to G/H$ the endomorphism induced by $\phi$. Then $$\ent^\star(\phi)=\ent^\star(\phi\restriction_H)+\ent^\star(\overline\phi).$$
\end{Theorem} 

The following is one of the main results on the adjoint algebraic entropy proved in \cite{DGS}. It shows that the adjoint algebraic entropy takes values only in $\{0,\infty\}$, while clearly the algebraic entropy may take also finite positive values.

\begin{Theorem}[Dichotomy Theorem]\label{dichotomy}\emph{\cite{DGS}}
Let $G$ be an abelian group and $\phi:G\to G$ an endomorphism. Then 
\begin{center}
either $\ent^\star(\phi)=0$ or $\ent^\star(\phi)=\infty$.
\end{center}
\end{Theorem}


Applying the Dichotomy Theorem and the Bridge Theorem (stated in the previous section) to the compact dual group $K$ of $G$ one gets that for a continuous endomorphism $\psi$ of a compact abelian group $K$ either $\ent (\psi)=0$ or $\ent(\psi)=\infty$. In other words:

\begin{Corollary}      
If $K$ is a compact abelian group, then every endomorphism $\psi:K\to K$ with $0 < \ent (\psi) < \infty$ is discontinuous.
\end{Corollary}

\subsection*{Acknowledgements}
It is a pleasure to thank our coauthor S. Virili for his kind permission to anticipate here some of  the main 
results from \cite{DGV1}. Thanks are due also to J. Spev\'ak for letting us insert his example in item (b) of Example \ref{ex-jan}, and to L. Busetti for the nice diagrams from his thesis \cite{LoBu} used in the present paper.

\end{document}